\documentclass[11pt]{article}
\usepackage[numbers,square]{natbib}
\usepackage{amsmath}
\usepackage{amsthm}
\usepackage{amsfonts}
\usepackage{amssymb}
\usepackage{siunitx}
\usepackage{commath}
\usepackage{graphicx}
\usepackage{xspace}
\usepackage{color}
\usepackage[lofdepth,lotdepth]{subfig}
\usepackage{stmaryrd}
\usepackage{url}
\usepackage{amsthm}
\usepackage{footnote}
\makesavenoteenv{tabular}
\makesavenoteenv{table}
\usepackage{tikz}
\usepackage[hidelinks]{hyperref}
\hypersetup{
  colorlinks   = true,   urlcolor     = red,   linkcolor    = red,   citecolor   = red }
\usepackage{cleveref}
\usepackage{booktabs}
\usepackage{multirow}

\newtheorem{lemma}{Lemma}[section]

\newtheorem{remark}{Remark}[section]

\usepackage[margin=0.7in]{geometry}
\makeatletter

\newcommand{\dnorm}{\@ifstar\@dnorms\@dnorm}
\newcommand{\@dnorms}[1]{  \left|\mkern-1.5mu\left|
  #1
  \right|\mkern-1.5mu\right|
}
\newcommand{\@dnorm}[2][]{  \mathopen{#1|\mkern-1.5mu#1|}
  #2
  \mathclose{#1|\mkern-1.5mu#1|}
}
\makeatother

\DeclareMathOperator{\atantwo}{atan2}

\title{AIR algebraic multigrid for a space-time hybridizable
  discontinuous Galerkin discretization of
  advection(-diffusion)\thanks{SR gratefully acknowledges
      support from the Natural Sciences and Engineering Research
      Council of Canada through the Discovery Grant program
      (RGPIN-05606-2015). BSS was supported by Lawrence Livermore
      National Laboratory under contracts B639443 and B634212, and as
      a Nicholas C. Metropolis Fellow under the Laboratory Directed
      Research and Development program of Los Alamos National
      Laboratory.  }}
\author{A. A. Sivas\thanks{Department of Applied Mathematics,
    University of Waterloo, Canada (\url{aasivas@uwaterloo.ca}),
    \url{http://orcid.org/0000-0002-5263-1889}} \and
  B. S. Southworth\thanks{Los Alamos National Laboratory, Los Alamos NM,
    U.S.A. (\url{southworth@lanl.gov}),
    \url{http://orcid.org/0000-0002-0283-4928}} \and
  S. Rhebergen\thanks{Department of Applied Mathematics, University of
    Waterloo, Canada (\url{srheberg@uwaterloo.ca}),
    \url{http://orcid.org/0000-0001-6036-0356}}}
\begin{document}
\maketitle
\begin{abstract}
  This paper investigates the efficiency, robustness, and scalability
  of approximate ideal restriction (AIR) algebraic multigrid as a
  preconditioner in the \textit{all-at-once} solution of a space-time
  hybridizable discontinuous Galerkin (HDG) discretization of
  advection-dominated flows. The motivation for this study is that the
  time-dependent advection-diffusion equation can be seen as a
  ``steady'' advection-diffusion problem in $(d+1)$-dimensions and AIR
  has been shown to be a robust solver for steady advection-dominated
  problems. Numerical examples demonstrate the effectiveness of AIR as
  a preconditioner for advection-diffusion problems on fixed and
  time-dependent domains, using both slab-by-slab and all-at-once
  space-time discretizations, and in the context of uniform and
  space-time adaptive mesh refinement. A closer look at the geometric
  coarsening structure that arises in AIR also explains why AIR can
  provide robust, scalable space-time convergence on advective and
  hyperbolic problems, while most multilevel parallel-in-time schemes
  struggle with such problems.
\end{abstract}
\section{Introduction}
\label{sec:introduction}
In this paper, we are interested in the fast parallel solution of the
time-dependent advection(-diffusion) problem on a time-dependent
domain $\Omega(t)$,
\begin{equation}
  \label{eq:advdiftimedepdomain}
  \partial_t u + a \cdot \nabla u - \nu \nabla^2 u = f
  \qquad \text{in } \Omega(t),\ t^0 < t < t^N,
\end{equation}
where $a$ is the advective velocity, $f$ is a source term, and $\nu\geq0$ is the
diffusion constant. We are particularly interested in the
advection-dominated regime where $0 \le \nu \ll 1$.

To discretize \cref{eq:advdiftimedepdomain}, we consider the space-time
framework in which the problem is recast into a space-time domain as
follows. Let $x=(x_1,\hdots,x_d)$ be the spatial variables in spatial
dimension $d$. A point at time $t=x_0$ with position $x$ then
has Cartesian coordinates $\hat{x}=(x_0,x)$ in space-time. Defining
the space-time domain
$\mathcal{E} := \cbr[0]{\hat{x}\ :\ x\in\Omega(x_0),\ t^0 < x_0 < t^N}$, the
space-time advective velocity $\widehat{a} := (1,a)$ and the
space-time gradient $\widehat{\nabla} := (\partial_t, \nabla)$, the
space-time formulation of \cref{eq:advdiftimedepdomain} is given by
\begin{equation}
  \label{eq:advdiftimedepdomain_st}
  \widehat{a} \cdot \widehat{\nabla} u - \nu \nabla^2 u = f \qquad \text{in } \mathcal{E}.
\end{equation}

There are multiple reasons to consider space-time finite element
methods over traditional discretizations. First, space-time
methods provide a natural framework for the discretization of partial
differential equations on time-dependent domains \cite{Hubner:2004,
  Masud:1997, Tezduyar:1992b, Tezduyar:2006, Walhorn:2005}. This is
because the domain and mesh movement are automatically accounted for by
the space-time finite element spaces, which are defined on a
triangulation of the space-time domain $\mathcal{E}$. Furthermore,
since there is no distinction between spatial and temporal variables,
it is relatively straightforward to allow local time stepping and
adaptive space-time mesh refinement (see, for example
\cite{Ven:2008}). This is particularly interesting from an efficiency
perspective for problems that require locally small time steps and fine
mesh resolution to achieve high levels of accuracy in only some parts of
the domain. These properties are non-trivial within the
context of traditional time-integration techniques. Finally,
space-time finite elements allow for greater parallelization by
solving for the entire space-time solution simultaneously, rather
than in a sequential time-stepping process. This ends up being
particularly relevant for hyperbolic PDEs, as will be discussed
later.

Space-time discontinuous Galerkin (DG) finite element methods are well
suited for solving \cref{eq:advdiftimedepdomain_st} in the
advection-dominated limit (see \cite{Rhebergen:2013b, Sollie:2011,
  Tavelli:2015, Tavelli:2016, Sudirham:2008, Vegt:2002, Wang:2015} and
references therein). This is because space-time DG methods incorporate
upwinding in their numerical fluxes, are locally conservative, and
automatically satisfy the geometric conservation law (GCL)
\cite{Lesoinne:1996}, which requires that the uniform flow remains uniform
under grid motion. We point out that alternative discretizations (such
as arbitrary Lagrangian--Eulerian methods) may require additional
constraints to satisfy the GCL \cite{Persson:2009}.  One downside of
space-time DG methods is the large number of globally coupled
degrees-of-freedom (DOFs) that arise when applying DG finite elements
in $(d+1)$-dimensional space. However, the space-time hybridizable
discontinuous Galerkin (HDG) method \cite{Rhebergen:2012,
  Rhebergen:2013}, introduced as a space-time extension of the HDG
method \cite{Cockburn:2009}, can attenuate this problem. The
space-time HDG method, like the HDG method, introduces approximate
traces of the solution on the element faces. The DOFs on the interior of
an element are then eliminated from the system, resulting in a
(significantly smaller) global system of algebraic equations only for
the approximate traces. However, it should be noted that a reduction
in the number of globally coupled DOFs does not necessarily imply a
more efficient time to solution -- the linear system still needs to be
solved.

In practice, a \emph{slab-by-slab} approach is almost exclusively used
to obtain the solution of space-time discretizations, which is
analogous to traditional time-integration techniques: the space-time
domain is partitioned into space-time slabs and local systems are
solved sequentially one time step after the other
(e.g. \cite{Jamet:1978, Masud:1997, Vegt:2002}). Although commonly
used, such an approach is limited to spatial parallelism, which
eventually plateaus in the sense that using more processors does not
speed up the time to the solution (see, e.g., \cite{Falgout:2014}). With
an increasing number of processors available for use and stagnating
core clock speeds, there has been significant research on
parallel-in-time (PinT) methods in recent years.

Some of the most effective PinT methods are multigrid-in-time methods,
where a parallel multilevel method is applied over the time domain,
which is then coupled with traditional spatial solves to perform time
steps of varying sizes (in particular, see Parareal \cite{Lions:2001}
and multigrid-reduction-in-time \cite{Falgout:2014}). Such methods are
effective on parabolic-type problems, but tend to not be robust or
just not convergent on advection-dominated and hyperbolic problems
without special treatment (for example, see
\cite{ruprecht2018wave,friedhoff2020optimal, de2020convergence, dai2013stable,
  de2019optimizing}). The simplest explanation for the difficulties
such methods have with hyperbolic problems is the separation of space
and time. By treating space and time separately, the multilevel
coarsening cannot respect the underlying characteristics that
propagate in space-time.

A more general approach is to consider space-time multigrid, that is,
multigrid methods applied to the full space-time domain. To our
knowledge, such an approach has only been applied to parabolic
problems, primarily the heat equation \cite{Weinzierl:2012,
  horton1995space, gander2016analysis}. However, even there,
space-time multigrid has demonstrated superior performance over PinT
methods that use multigrid in space and time separately
\cite{falgout2017multigrid}.  Recently, auxiliary-space
preconditioning techniques have also been proposed for space-time
finite-element discretizations \cite{Gopalakrishnan:2018}, which has
the potential to provide more general space-time solvers. Continuing with
the above discussion, the \emph{all-at-once} approach to space-time
finite elements constructs and solves a single global linear system
for the solution in the whole space-time domain. From a solver's
perspective, we claim that the all-at-once approach is particularly
well suited for advective and hyperbolic problems.

The main contribution of this paper is demonstrating the suitability
of the nonsymmetric algebraic multigrid (AMG) method based on
Approximate Ideal Restriction (AIR) \cite{Manteuffel:2019,
  Manteuffel:2018} for the solution of slab-by-slab and, in
particular, all-at-once space-time HDG discretizations of the
advection-diffusion problem in advection-dominated
regimes. Advection-dominated problems are typically difficult to solve
due to the non-symmetric nature of the problem. Nevertheless,
significant developments in multigrid methods for non-symmetric
problems have been made in recent years \cite{Notay:2012,
  Wiesner:2014, Brezina:2010dm, Sala:2008cv, Manteuffel:2017}.  In
particular, AIR has shown to be a robust solver for steady
advection-dominated problems. This motivates us to study AIR for a
space-time HDG discretization of the advection-diffusion problem,
since \cref{eq:advdiftimedepdomain_st} can be seen as a ``steady''
advection-diffusion problem in $(d+1)$-dimensions.

The remainder of this paper is organised as follows. In
\cref{sec:problem}, we present the space-time HDG discretization of the
advection-diffusion equation, and AIR
is presented in \cref{sec:preconditioning}. A discussion on why AIR
can be effective as a space-time solver of advection-dominated
problems, while most PinT methods struggle, is provided in \cref{ss:coarsen}.
Numerical results in \cref{sec:results} indeed demonstrate that
AIR is a robust and scalable solver for space-time HDG discretizations
of the advection-diffusion equation. Scalable preconditioning is
demonstrated with space-time adaptive mesh refinement (AMR) and on
time-dependent domains, and speedups over sequential time
stepping are obtained on very small processor counts.
We draw conclusions in \cref{sec:conclusions}.

\section{The space-time HDG method for the advection-diffusion equation}
\label{sec:problem}

\subsection{The advection-diffusion problem on time-dependent domains}
\label{ss:advdif}

Let $\Omega_h(t) \subset \mathbb{R}^d$, an approximation to the domain
$\Omega(t)$ in \cref{eq:advdiftimedepdomain}, be a polygonal ($d=2$)
or polyhedral ($d=3$) domain whose evolution depends continuously on
time $t \in \sbr[0]{t^0, t^N}$. We will present numerical results only
for the case $d=2$, but remark that the space-time HDG discretization
and solution procedure also hold for $d=3$. We partition the boundary
of $\Omega_h(t)$, $\partial\Omega_h(t)$, into two sets $\Gamma_D(t)$
(the Dirichlet boundary) and $\Gamma_N(t)$ (the Neumann boundary) such
that $\partial \Omega_h(t) = \Gamma_D(t) \cup \Gamma_N(t)$ and
$\Gamma_D(t) \cap \Gamma_N(t) = \emptyset$.

As discussed in \cref{sec:introduction}, a point in space-time at time
$t=x_0$ with position $x$ has Cartesian coordinates
$\hat{x}=(x_0,x)$. Throughout this paper, we will use $t$ and $x_0$
interchangeably. We introduce the $(d+1)$-dimensional computational
space-time domain
$\mathcal{E}_h := \cbr[0]{\hat{x}\, :\, x \in \Omega_h(x_0),\ t^0 <
  x_0 < t^N} \subset \mathbb{R}^{d+1}$. The boundary of
$\mathcal{E}_h$ is comprised of the hyper-surfaces
$\Omega_h(t^0) := \cbr[0]{\hat{x} \in \partial\mathcal{E}_h\, :\,
  x_0=t^0}$,
$\Omega_h(t^N) := \cbr[0]{\hat{x} \in \partial\mathcal{E}_h\, :\,
  x_0=t^N}$, and
$\mathcal{Q}_{\mathcal{E}_h} := \cbr[0]{\hat{x} \in
  \partial\mathcal{E}_h\, : \, t^0 < x_0 < t^N}$. We also introduce
the partitioning
$\partial\mathcal{E}_h = \partial\mathcal{E}_D \cup
\partial\mathcal{E}_N$ where
$\partial\mathcal{E}_D := \cbr[0]{\hat{x} \, : \, x \in
  \Gamma_D(x_0),\ t^0
  < x_0 < t^N}$ and \\
$\partial\mathcal{E}_N := \cbr[0]{\hat{x} \, : \, x \in \Gamma_N(x_0)
  \cup \Omega(t^0),\ t^0 < x_0 \le t^N}$. The outward unit space-time
normal vector to $\partial \mathcal{E}_h$ is denoted by
$\widehat{n}=(n_t, n)$, where $n_t \in \mathbb{R}$ is the temporal
part of the space-time vector and $n \in \mathbb{R}^d$ the spatial
part.

Given the viscosity $\nu \geq 0$, forcing term
$f : \mathcal{E}_h \to \mathbb{R}$, and advective velocity
$a : \mathcal{E}_h \to \mathbb{R}^d$, the advection-diffusion equation
for the scalar $u : \mathcal{E}_h \to \mathbb{R}$ is given by
\begin{subequations}
  \begin{align}
    \label{eq:advdif_a}
    \partial_t u + a \cdot \nabla u - \nu \nabla^2 u &= f  && \text{in } \mathcal{E}_h,
    \\
    \label{eq:advdif_b}
    -\zeta u (n_t + a \cdot n) + \nu \nabla u \cdot n &= g_N && \text{on } \partial\mathcal{E}_N,
    \\
    \label{eq:advdif_c}
    u &= g_D && \text{on } \partial\mathcal{E}_D,
  \end{align}
  \label{eq:advdif}
\end{subequations}
where $g_N :\mathcal{Q}_N \to \mathbb{R}$ is a suitably smooth
function and $\zeta$ is an indicator function for the inflow boundary
of $\mathcal{E}$, i.e., where $(n_t + a \cdot n) < 0$. Note that the
initial condition $u(0,x) = g_N(0,x)$ is imposed by
\cref{eq:advdif_b}. Using the definition of the space-time advective
velocity and the space-time gradient introduced in
\cref{sec:introduction}, the space-time formulation of
\cref{eq:advdif} is given by
\begin{subequations}
  \begin{align}
    \label{eq:stadvdif_a}
    \widehat{a} \cdot \widehat{\nabla} u - \nu \nabla^2 u &= f  && \text{in } \mathcal{E}_h,
    \\
    \label{eq:stadvdif_b}
    -\zeta u \widehat{a}_n + \nu \nabla u \cdot n &= g_N && \text{on } \partial\mathcal{E}_N,
    \\
    \label{eq:stadvdif_c}
    u &= g_D && \text{on } \partial\mathcal{E}_D,
  \end{align}
  \label{eq:stadvdif}
\end{subequations}
where
$\widehat{a}_n = \widehat{n} \cdot \widehat{a} = n_t + a\cdot n$. We
see that the time-dependent advection-diffusion problem
\cref{eq:advdif} is a steady state problem in $(d+1)$-dimensional
space-time.

\subsection{Space-time meshes}
\label{ss:st-meshes}
The two approaches to meshing a space-time domain $\mathcal{E}_h$ are
the slab-by-slab approach and the all-at-once approach.  In the
\emph{slab-by-slab} approach, the time interval $[t^0, t^N]$ is
partitioned into time levels $t^0 < t^1 < \cdots < t^N$. The $n$-th
time interval is defined as $I^n = (t^n, t^{n+1})$ and its length is
the ``time-step'', denoted by $\Delta t^n = t^{n+1} - t^n$. The
space-time domain $\mathcal{E}_h$ is then divided into space-time
slabs
$\mathcal{E}_h^n = \mathcal{E}_h \cap (I^n \times \mathbb{R}^d)$. Note
that each space-time slab $\mathcal{E}_h^n$ is bounded by
$\Omega_h(t^n)$, $\Omega_h(t^{n+1})$, and
$\mathcal{Q}_{\mathcal{E}_h}^n = \partial\mathcal{E}_h^n \backslash
(\Omega_h(t^n) \cup \Omega_h(t^{n+1}))$. A space-time triangulation
$\mathcal{T}_h^n$ is then introduced for each space-time slab
$\mathcal{E}_h^n$ using standard spatial meshing techniques. In this
paper, we use space-time simplices (see, e.g. \cite{Horvath:2019,
  Horvath:2020, Wang:2015}) as opposed to space-time hexahedra (see
e.g. \cite{Ambati:2007, Vegt:2002, Ven:2008}).

In the \emph{all-at-once} approach, a space-time triangulation
$\mathcal{T}_h := \cup_j\mathcal{K}_j$ of the full space-time domain
$\mathcal{E}_h$ is introduced. This triangulation consists of
non-overlapping space-time simplices
$\mathcal{K} \subset \mathbb{R}^{d+1}$. There are no clear time levels
except for the time level at $x_0=t^0$ and $x_0=t^N$ and the
space-time mesh may be fully unstructured. In particular, this
naturally allows for arbitrary adaptive mesh refinement (AMR) in space
and time. Note, we do not consider hanging nodes in this paper
although hanging nodes in space and time are possible within the
space-time framework.

In \cref{fig:ST_elements} we plot space-time elements in a
slab-by-slab approach and in an all-at-once approach in
$(1+1)$-dimensional space-time.

\begin{figure}[h!]
  \centering
  \begin{tikzpicture}[scale=0.75]
        \draw[->, thick] (0,-2.5)--(0,2.5) node[above]{$x_0$};
    \draw[->, thick] (0,-2.5)--(5,-2.5) node[right]{$x$};
    \draw[dashed] (0, -1)--(5, -1);
    \draw(6, -1) node{$\Omega_h(t^n)$};
    \draw[dashed] (0, 1)--(5, 1);
    \draw(6,  1) node{$\Omega_h(t^{n+1})$};
    \draw[thick] (1.5, -1)--(3.5, -1);
    \draw[thick] (2, 1)--(3.5, -1);
    \draw[thick] (2, 1)--(3, 1);
    \draw[thick] (1.5, -1)--(2, 1);
    \draw[thick] (3.5, -1)--(3, 1);
    \draw(-0.7, -1) node{$t^n$};
    \draw(-0.7, 1) node{$t^{n+1}$};
    \draw(2.5, -0.5) node{$\mathcal{K}_{l}$};
    \draw(2.8, 0.5) node{$\mathcal{K}_{j}$};
    \draw[->] (2.5, 1.6)--(2.5, 1.1);
    \draw(2.5, 2.0) node{$K_j^{n+1}$};
    \draw[->] (2.5, -1.6)--(2.5, -1.1);
    \draw(2.5, -2.0) node{$K_l^{n}$};
    \draw[->, thick] (10,-2.5)--(10,2.5) node[above]{$x_0$};
    \draw[->, thick] (10,-2.5)--(15,-2.5) node[right]{$x$};
    \draw[thick] (11.5, -1)--(13.5, -0.5);
    \draw[thick] (12, 1)--(13.5, -0.5);
    \draw[thick] (12, 1)--(13, 1.5);
    \draw[thick] (11.5, -1)--(12, 1);
    \draw[thick] (13.5, -0.5)--(13, 1.5);
    \draw(12.3, -0.2) node{$\mathcal{K}_l$};
    \draw(12.8, 0.7) node{$\mathcal{K}_j$};
      \end{tikzpicture}
  \caption{Examples of two neighboring elements in $(1+1)$-dimensional
    space-time.  Left: An example of space-time elements in a
    slab-by-slab approach. The space-time mesh is layered by
    space-time slabs. Here the elements lie in space-time slab
    $\mathcal{E}_h^n$. Right: An example of space-time elements in an
    all-at-once approach. There are no clear time levels for
    $t^0 < x_0 < t^N$.}
  \label{fig:ST_elements}
\end{figure}
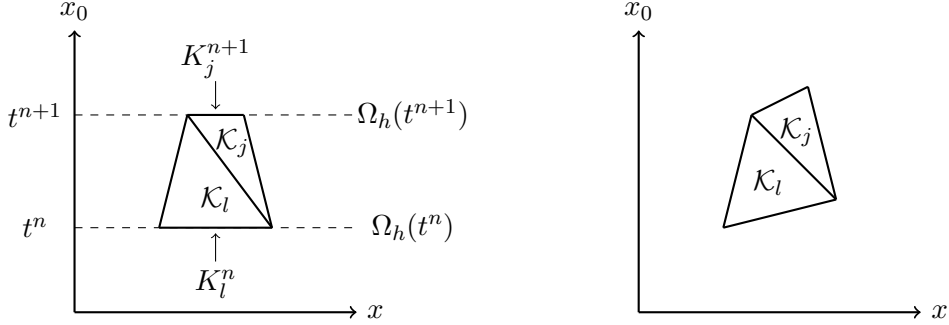

\subsection{The space-time HDG method}
\label{ss:sthdg}

Consider a space-time element $\mathcal{K} \in \mathcal{T}_h$ in an
all-at-once or slab-by-slab mesh. On the boundary of a space-time
element $\partial\mathcal{K}$ we will denote the outward unit
space-time normal vector by
$\widehat{n}^{\mathcal{K}} = (n_t^{\mathcal{K}},
n^{\mathcal{K}})$. Two adjacent space-time elements $\mathcal{K}^+$
and $\mathcal{K}^-$ share an interior space-time facet
$\mathcal{S} := \partial\mathcal{K}^+ \cap \partial\mathcal{K}^-$. A
facet of $\partial\mathcal{K}$ that lies on the space-time boundary
$\partial\mathcal{E}_h$ is called a boundary facet. The set of all
facets is denoted by $\mathcal{F}$ and the union of all facets by
$\Gamma_0$. For ease of notation, we will drop the subscripts and
superscripts when referring to space-time elements, their boundaries,
and outward unit normal vectors in the remainder of this article.

We require the following finite element spaces:
\begin{align*}
  V_h &:= \cbr[0]{v_h \in L^2(\mathcal{E}_h) \, :\, v_h|_{\mathcal{K}}
        \in P_p(\mathcal{K}),\ \forall\mathcal{K} \in \mathcal{T}_h},
  \\
  M_h &:= \cbr[0]{\mu_h \in L^2(\mathcal{F}) \, : \, \mu_h|_{\mathcal{S}}
        \in P_p(\mathcal{S}),\ \forall \mathcal{S} \in \mathcal{F},\
        \mu_h=0 \text{ on } \partial\mathcal{E}_D},
\end{align*}
where $P_p(D)$ is the set of polynomials of degree $p$ on a domain
$D$. We furthermore introduce $V_h^{\star} := V_h \times M_h$. The
space-time HDG method for \cref{eq:stadvdif} is given by
\cite{Kirk:2019}: find $(u_h, \lambda_h) \in V_h^{\star}$ such that
\begin{equation}
  \label{eq:st-hdg}
  \mathcal{B}_h\del[1]{(u_h, \lambda_h), (v_h, \mu_h)}
  = \sum_{\mathcal{K}\in\mathcal{T}_h} \int_{\mathcal{K}} f v_h \dif \hat{x}
  + \int_{\partial\mathcal{E}_N} g \mu_h \dif s
  \quad \forall (v_h, \mu_h) \in V_h^{\star},
\end{equation}
where the bilinear form is defined as
\begin{equation}
  \label{eq:bilin}
  \begin{split}
    \mathcal{B}_h\big( (u, \lambda), & (v, \mu) \big)
    \\
    :=& \sum_{\mathcal{K} \in \mathcal{T}_h}
    \int_{\mathcal{K}} \del[1]{ - u\widehat{a}\cdot\widehat{\nabla}v + \nu \nabla u \cdot \nabla v} \dif \hat{x}
    + \int_{\partial\mathcal{E}_N} \tfrac{1}{2}\del[0]{ \widehat{a}_{n} + |\widehat{a}_{n}| } \lambda\mu \dif s
    \\
    &+ \sum_{\mathcal{K} \in \mathcal{T}_h} \int_{\partial\mathcal{K}} \sigma(u,\lambda,\widehat{n}) (v-\mu) \dif s
    - \sum_{\mathcal{K} \in \mathcal{T}_h} \int_{\partial\mathcal{K}} \nu(u-\lambda)\nabla v \cdot n \dif s.
  \end{split}
\end{equation}
Here
$\sigma(u,\lambda,\widehat{n}) := \sigma_a(u,\lambda,\widehat{n}) +
\sigma_d(u,\lambda,n)$ is the ``numerical flux'' on the cell
facets. The advective part of the numerical flux is an upwind flux in
both space and time, given by
\begin{equation*}
  \sigma_a(u,\lambda,\widehat{n})
  := \tfrac{1}{2}\del[1]{ \widehat{a}_{n}(u + \lambda)
    + |\widehat{a}_{n}|(u-\lambda)}.
\end{equation*}
The diffusive part of the numerical flux is similar to that of an
interior penalty method and is given by
\begin{equation}
  \label{eq:IPpenalty}
  \sigma_d(u,\lambda,n)
  := -\nu\nabla u \cdot n + \frac{\nu\alpha}{h_{\mathcal{K}}}(u-\lambda),
\end{equation}
with $h_{\mathcal{K}}$ the length measure of the element
$\mathcal{K}$, and $\alpha > 0$ a penalty parameter. It is shown in
\cite{Kirk:2019} that $\alpha$ needs to be sufficiently large to
ensure stability of the space-time HDG method.

\subsection{Sequential time-stepping using the slab-by-slab
  discretization}

The space-time HDG method \cref{eq:st-hdg} is the same for both the
slab-by-slab and all-at-once space-time approaches. However, for the
slab-by-slab approach we may write \cref{eq:st-hdg} in a form similar
to traditional time-integration techniques. For this we require the
following finite element spaces:
\begin{align*}
  V_h^n &:= \cbr[0]{v_h \in L^2(\mathcal{E}_h^n) \, :\, v_h|_{\mathcal{K}} \in P_p(\mathcal{K}),\
          \forall\mathcal{K} \in \mathcal{T}_h^n},
  \\
  M_h^n &:= \cbr[0]{\mu_h \in L^2(\mathcal{F}^n) \, : \, \mu_h|_{\mathcal{S}} \in P_p(\mathcal{S}),\
          \forall \mathcal{S} \in \mathcal{F}^n,\ \mu_h=0 \text{ on } \partial\mathcal{E}_D^n},
\end{align*}
where $\mathcal{F}^n$ is the set of all facets in the slab
$\mathcal{E}_h^n$. We furthermore define
$V_h^{n, \star} := V_h^n \times M_h^n$. For the slab-by-slab approach,
we may write the space-time HDG method for \cref{eq:advdif} as: for
each space-time slab $\mathcal{E}_h^n$, $n=0,1,\cdots,N-1$, find
$(u_h, \lambda_h) \in V_h^{n, \star}$ such that
\begin{equation}
  \label{eq:st-hdg-sbs}
  \mathcal{B}_h^n\del[1]{(u_h, \lambda_h), (v_h, \mu_h)}
  = \sum_{\mathcal{K}\in\mathcal{T}_h^n} \int_{\mathcal{K}} f v_h \dif \hat{x}
  + \int_{\partial\mathcal{E}_N^n} g \mu_h \dif s,
\end{equation}
for all $(v_h, \mu_h) \in V_h^{n, \star}$, where
$\mathcal{B}_h^n(\cdot, \cdot)$ is defined as \cref{eq:bilin} but with
$\mathcal{T}_h$ and $\partial\mathcal{E}_N$ replaced by, respectively,
$\mathcal{T}_h^n$ and $\partial\mathcal{E}_N^n$. The slab-by-slab
approach is similar to traditional time-integration techniques in that
the local systems are solved one space-time slab after another. The linear
systems arising from space-time finite elements
resemble those that arise from fully implicit Runge--Kutta methods
(e.g., see \cite{makridakis2006posteriori, schotzau2000time}).

Well-posedness and convergence of the slab-by-slab space-time HDG
method \cref{eq:st-hdg-sbs} was proven in
\cite{Kirk:2019}. Furthermore, motivated by the fact that the spatial
mesh size $h_K$ and the time-step $\Delta t$ may be different, an a
priori error analysis was presented in \cite{Kirk:2019}, resulting in
optimal error bounds that are anisotropic in $h_K$ (a measure of the
mesh size in
spatial direction) and $\Delta t$. It is shown, however, that $\Delta t$
and $h_K$ need to be refined simultaneously to obtain these optimal
error bounds, and that refining only in time or only in space may lead
to divergence of the error. To this end, all-at-once solvers seem like
the natural solution for efficient parallel simulations, where
simultaneous local adaptivity in space and time is easily handled.

\subsection{The discretization}
\label{ss:discretization}
Let $U \in \mathbb{R}^r$ be the vector of expansion coefficients of
$u_h$ with respect to the basis for $V_h$ and let
$\Lambda \in \mathbb{R}^q$ be the vector of expansion coefficients of
$\lambda_h$ with respect to the basis for $M_h$. The space-time HDG
method \cref{eq:st-hdg} can then be expressed as the all-at-once
system of linear equations
\begin{equation}
  \label{eq:linsys}
  \begin{bmatrix}
    A & B \\ C & D
  \end{bmatrix}
  \begin{bmatrix}
    U \\ \Lambda
  \end{bmatrix}
  =
  \begin{bmatrix}
    F \\ G
  \end{bmatrix},
\end{equation}
where $A$, $B$, $C$, and $D$ are matrices obtained from the
discretization of $\mathcal{B}_h((\cdot, 0), (\cdot, 0))$,
$\mathcal{B}_h((0,\cdot ), (\cdot, 0))$,
$\mathcal{B}_h((\cdot, 0), (0, \cdot))$, and
$\mathcal{B}_h((0, \cdot ), (0, \cdot))$, respectively.

For the slab-by-slab approach, the linear system \cref{eq:linsys} can
be decoupled into smaller linear systems that are solved in each time
slab $\mathcal{E}_h^n$. In this case, $U \in \mathbb{R}^r$ is the
vector of expansion coefficients of $u_h$ with respect to the basis
for $V_h^n$ and $\Lambda \in \mathbb{R}^q$ is the vector of expansion
coefficients of $\lambda_h$ with respect to the basis for
$M_h^n$. Furthermore, $A$, $B$, $C$, and $D$ are then the matrices obtained
from the discretization of $\mathcal{B}_h^n((\cdot, 0), (\cdot, 0))$,
$\mathcal{B}_h^n((0,\cdot ), (\cdot, 0))$,
$\mathcal{B}_h^n((\cdot, 0), (0, \cdot))$, and
$\mathcal{B}_h^n((0, \cdot ), (0, \cdot))$, respectively.

The space-time HDG discretization is such that $A$ is a block-diagonal
matrix. Using $U = A^{-1}(F-B\Lambda)$ we eliminate $U$ from
\cref{eq:linsys} resulting in the following reduced system for
$\Lambda$:
\begin{equation}
  \label{eq:systemLambda}
  S\Lambda = H,
\end{equation}
where $S = D -CA^{-1}B$ is the Schur complement of the block matrix
in \cref{eq:linsys}, and $H=G-CA^{-1}F$. Having eliminated the element
degrees-of-freedom via static condensation, the linear system
\cref{eq:systemLambda} is significantly smaller than
\cref{eq:linsys}. However, for the space-time HDG method to be
efficient, we still require a fast solver for the reduced non-symmetric
problem \cref{eq:systemLambda}, which is discussed in the following
section.

\section{Approximate ideal restriction (AIR) AMG}
\label{sec:preconditioning}

AMG is traditionally designed for elliptic
problems in space or sequential time stepping of parabolic problems,
where the resulting linear systems are (nearly) symmetric positive
definite or M-matrices. However, a number of papers in recent
years have considered extensions of AMG to the nonsymmetric setting,
e.g., \cite{Notay:2012, Wiesner:2014, Brezina:2010dm, Sala:2008cv,
  Manteuffel:2017}. In particular, a new AMG method based on a local
approximate ideal restriction ($\ell$AIR; moving forward we simply
refer to it as AIR) was developed in
\cite{Manteuffel:2019, Manteuffel:2018} specifically for
advection-dominated problems and upwinded discretizations.
Noting that \cref{eq:stadvdif} is a ``steady''
advection-dominated problem in $(d+1)$-dimensional space-time, and that
AIR is a robust solver for advection dominated problems, motivates
the use of AIR as a preconditioner for the space-time linear system
\cref{eq:systemLambda}.

As a brief review, recall that multigrid methods solve
$A\mathbf{x} =\mathbf{b}$ by applying a coarse-grid correction based
on interpolation and restriction operators,
$\mathbf{x}^{(i+1)} = \mathbf{x}^{(i)} +
P(RAP)^{-1}R\mathbf{r}^{(i)}$, for matrix $A$, interpolation $P$,
restriction $R$, and residual
$\mathbf{r}^{(i)} = \mathbf{b} -A\mathbf{x}^{(i)}$. Classical AMG is
based on a partitioning of DOFs into fine (F-) and coarse (C-) points,
where $A$ can then be expressed in block form as
\begin{align*}
  A=\begin{bmatrix}
    A_{ff} & A_{fc} \\
    A_{cf} & A_{cc}
  \end{bmatrix}.
\end{align*}
AIR is a reduction method based on the principle that if we use
the so-called ideal restriction operator,
$R_{\text{ideal}} = \sbr[0]{-A_{cf}A_{ff}^{-1} \quad I}$ with any
interpolation (in MATLAB notation) $P = [W; I]$, coarse-grid correction
eliminates all errors at C-points; following this with an effective
relaxation on F-points will guarantee a rapidly convergent method
\cite[Section 2.3]{Manteuffel:2019}.
Due to the $A_{ff}^{-1}$ term in $R_{\text{ideal}}$, it is not practical to
form $R_{\text{ideal}}$ explicitly. However, AIR appeals to the
observation that for upwinded advective discretizations, one can achieve
cheap, accurate, and sparse approximations $R \approx R_{\text{ideal}}$.

\subsection{Coarsening in space-time}\label{ss:coarsen}

For problems with strong anisotropy or advective components, it is
often helpful or even necessary to semi-coarsen along the direction of
anisotropy/advection for an effective multigrid method (e.g.,
\cite{Wesseling:2001}). On a high-level, we claim that
one of the primary difficulties in applying common (multilevel) PinT
schemes to advective/hyperbolic problems is the separate treatment of
temporal and spatial variables. A natural result of this is that
coarsening performed separately in space and time is often unable to
align with hyperbolic characteristics in space-time.  Conversely, by
treating space and time all-at-once, it is natural for coarsening to
align with characteristics, which provides an important piece of a
scalable multilevel method.

\Cref{fig:CFpoints} demonstrates how classical AMG coarsening
\cite{Ruge:1987} applied to a hyperbolic 2d-space/1d-time HDG
discretization naturally applies semi-coarsening along the direction
of (space-time) characteristics. For clarity, examples are shown in
two-dimensional subdomains for the problem described in
\cref{ss:efficiencyZZ}, with plots for the advective field and the
corresponding CF-splitting.  The velocity fields given in (a) and (b)
correspond to CF-splittings in (d) and (e), respectively. Note that
for both cases, we largely see stripes of fine and coarse points
orthogonal to the flow direction, which is exactly semi-coarsening
along the characteristics. Similarly, in (c), note that in the
$[0,0.2]\times[0,0.2]$ spatial subdomain (for all time), there is
effectively no spatial advection, and thus the space-time advective
field is only traveling forward in time. Plots (f)--(i) demonstrate an
effective semi-coarsening in time, where we mark coarse points on the
time levels (see (f) and (i)) and fine points on interior time DOFs
(see (g) and (h)).

\begin{figure}[!htb]
  \centering
  \subfloat[Velocity field, $x=0.1861$]{
    \includegraphics[width=0.30\textwidth]{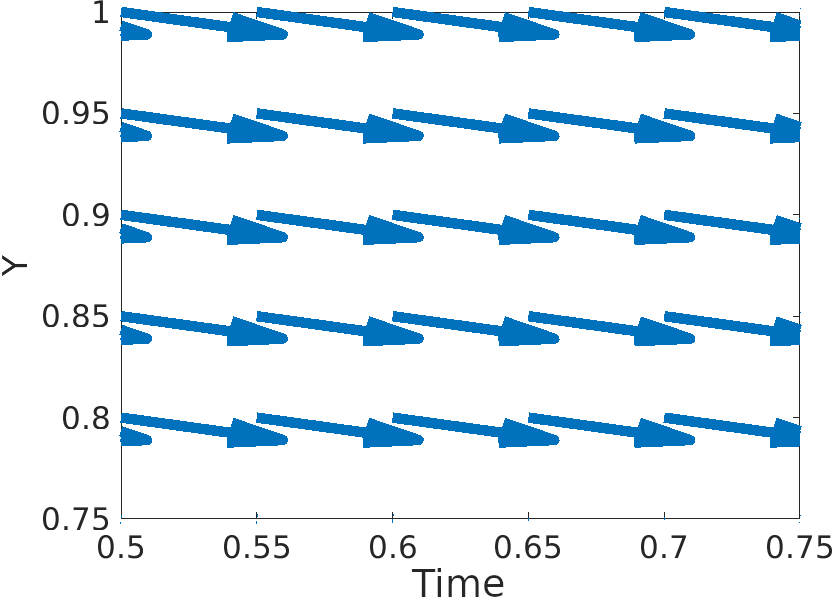}
  }
  \subfloat[Velocity field, $x=0.5790$]{
    \includegraphics[width=0.30\textwidth]{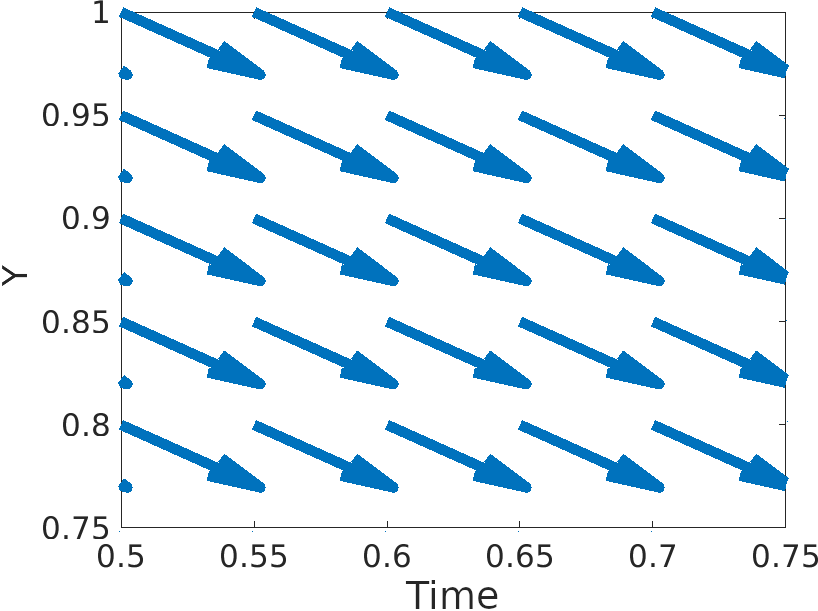}
  }
  \subfloat[Velocity field, any fixed $t$]{
    \includegraphics[width=0.30\textwidth]{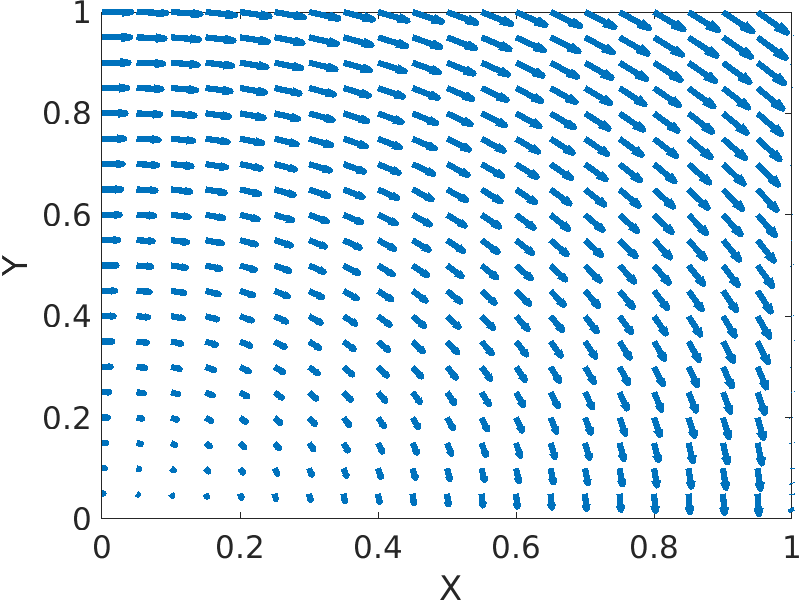}
  }
  \\
  \subfloat[$x=0.1861$]{
    \includegraphics[width=0.30\textwidth]{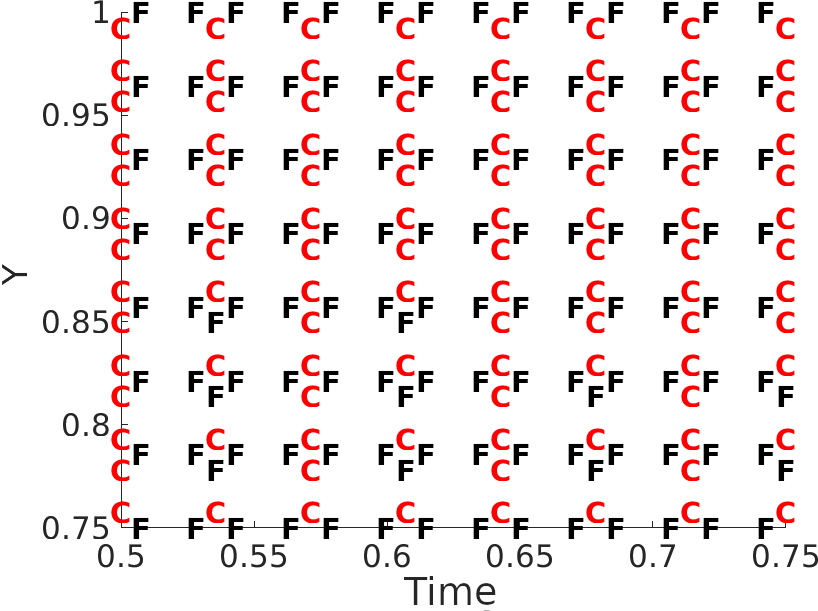}
  }
  \subfloat[$x=0.5790$]{
    \includegraphics[width=0.30\textwidth]{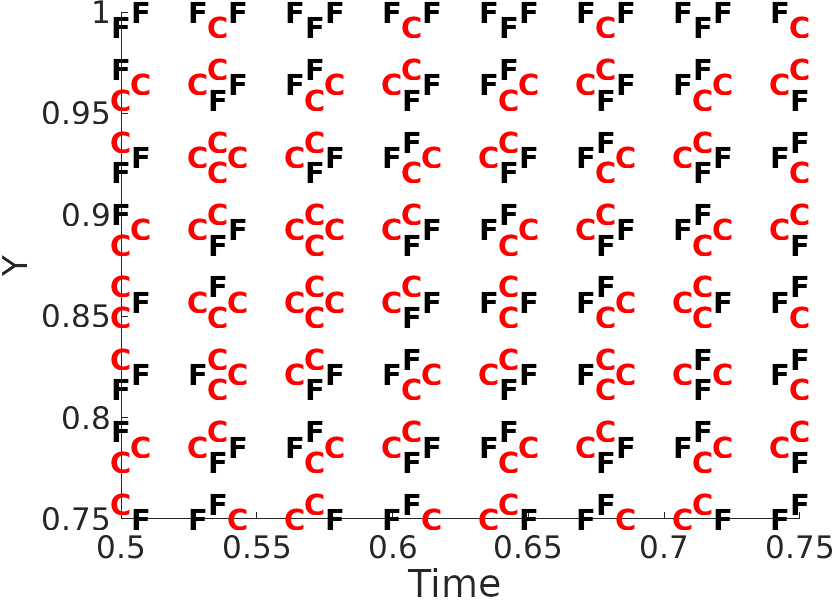}
  }
  \subfloat[$t=0.3571$]{
  \includegraphics[width=0.30\textwidth]{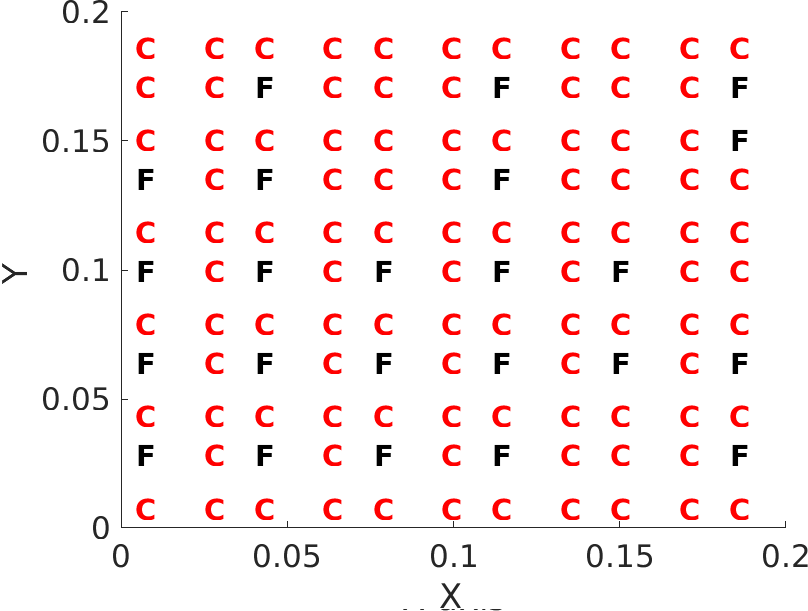}
  }
  \\
  \subfloat[$t=0.3646$]{
  \includegraphics[width=0.30\textwidth]{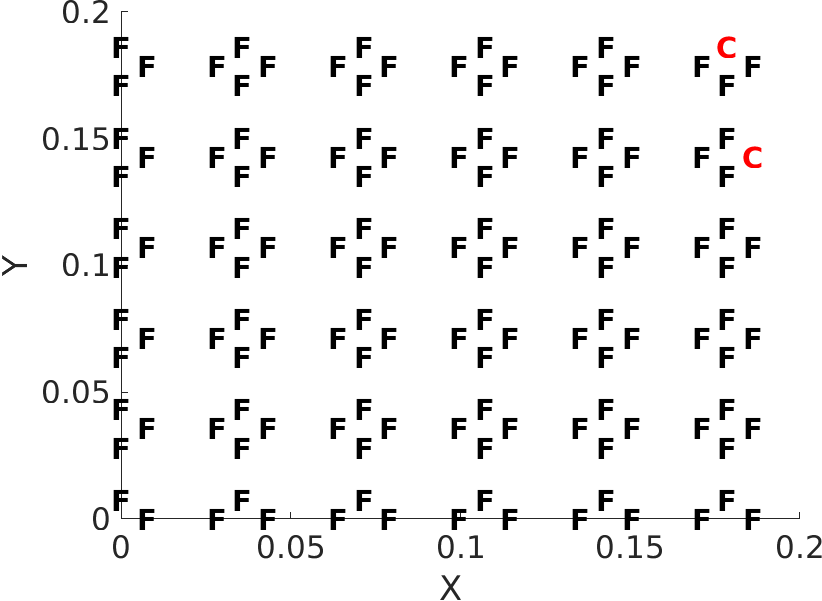}
  }
  \subfloat[$t=0.3853$]{
  \includegraphics[width=0.30\textwidth]{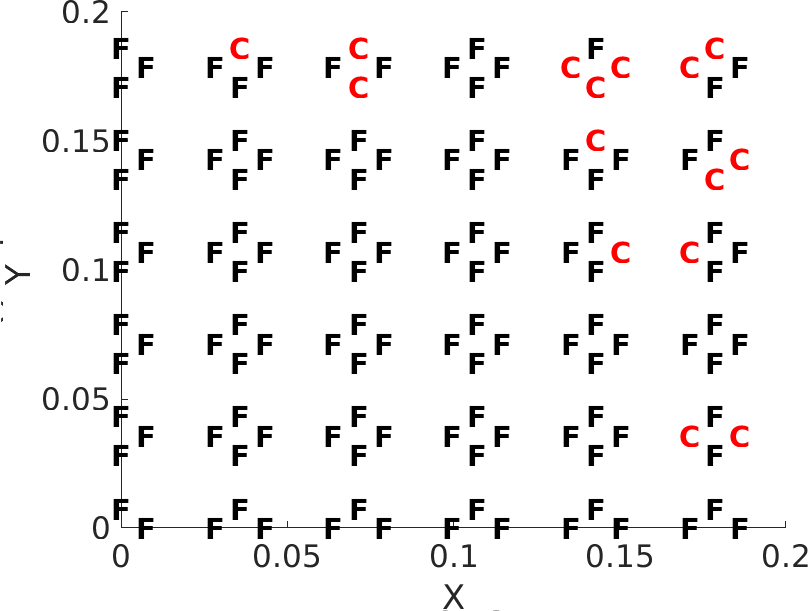}
  }
  \subfloat[$t=0.3928$]{
  \includegraphics[width=0.30\textwidth]{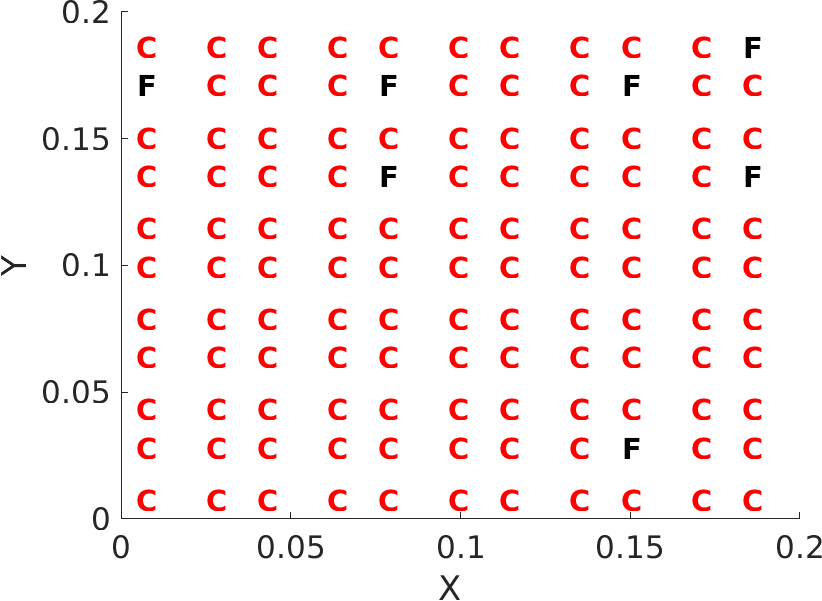}
  }
        \caption{Red points are C-points and black points are F-points
          for the hyperbolic problem
          from~\cref{ss:efficiencyZZ}. Distribution of the C- and F-
          points follow the velocity fields, showing semi-coarsening
          along characteristics.}
  \label{fig:CFpoints}
\end{figure}

\subsection{Relaxation and element ordering}
\label{ss:ordering}

Upwinded discontinuous discretizations of linear hyperbolic problems
have the benefit that the mesh elements can typically be reordered to
be (element) block lower triangular. The corresponding linear system
can then be solved directly using a forward solve. Although this is
not scalable in parallel (because each process must wait for the
previous to finish its solve), it provides an excellent relaxation
scheme when each core inverts the subdomain stored
on-process. Such a method is commonly used in the transport
community to avoid the parallel cost and complexity of a full forward
solve, and was shown to provide strong convergence when used with AIR
in \cite{hanophy2020parallel}.

Here, we show that an analogous property holds for HDG discretizations
of advection (steady or space-time). We do
so by noting that the existence of such an ordering is equivalent to
proving that the graph of the discretization matrix is acyclic. Assume
all mesh elements are convex, let $\{\mathcal{K}_i\}$ denote the set of
all $n_e$ elements for a given mesh, and let
$\mathbb{E}$ denote the graph of connections between elements, where
$\mathbb{E}_{ij} = 1$ if $\exists$ a connection from
$\mathcal{K}_i\mapsto \mathcal{K}_j$ with
respect to the given velocity field and $\mathbb{E}_{ij} = 0$ otherwise. Let
$\{\mathcal{S}_{ij}\}$ denote the set of all $n_f$ outgoing faces, with the
subscript $\mathcal{S}_{ij}$ indicating a connection
$\mathcal{K}_i\mapsto \mathcal{K}_j \in \mathbb{E}$, and $\mathbb{F}$
denoting the graph of connections between faces. Moreover note that
\begin{equation}
  \label{eq:f_iff}
  \mathcal{S}_{ij} \mapsto \mathcal{S}_{jk} \hspace{3ex} \textnormal{if and only if}
  \hspace{3ex} \mathcal{K}_i\mapsto \mathcal{K}_j \hspace{1ex}
  \textnormal{and}\hspace{1ex} \mathcal{K}_j\mapsto \mathcal{K}_k.
\end{equation}

\begin{lemma}
  \label{lem:order}
  Suppose $\mathbb{E}$ is a directed acyclic graph, and the elements
  $\{\mathcal{K}_i\}$ are ordered such that $\mathbb{E}$ is lower
  triangular. Furthermore, numerate
    faces $\mathcal{S}_{ij}$ with respect
  to index $i$ and then $j$, for example, $\{\mathcal{S}_{01},
  \mathcal{S}_{02},\mathcal{S}_{12},\mathcal{S}_{23},...\}$. Then,
  $\mathbb{F}$ is also a directed acyclic graph and lower triangular
  in this ordering.
\end{lemma}
\begin{proof}
  Because $\mathbb{E}$ is lower triangular, $\not\exists$
  $\mathcal{K}_i\mapsto \mathcal{K}_j$ if $i > j$. It follows by definition that $i < j$
  for all faces $\mathcal{S}_{ij}$. Now, suppose there exists a path
  $\mathcal{S}_{ij}\mapsto \mathcal{S}_{jk}$ in $\mathbb{F}$ such that
  $i > k$. By \cref{eq:f_iff}, this is true if and only if
  $\mathcal{K}_i\mapsto \mathcal{K}_j\mapsto \mathcal{K}_k$. However, this is a contradiction to
  the assumption of $\mathbb{E}$ being lower triangular. In addition note that by the convexity of elements,
  $\not\exists$ connections $\mathcal{S}_{ij}\mapsto \mathcal{S}_{ik}$, that is,
  connections between outgoing faces with respect to the velocity
  field on the same element. Enumerating $\{\mathcal{S}_{ij}\}$ first by
  index $i$, then (arbitrarily) by index $j$ as the set of faces
  $\{\widehat{\mathcal{S}}_\ell\}$
  implies $\not\exists$ path $g_{ij}\in\mathbb{F}$,
  $g:\widehat{\mathcal{S}}_i\mapsto \widehat{\mathcal{S}}_j$,
  such that $i > j$, which completes the proof.
\end{proof}

\Cref{lem:order} is useful in that an ordering can be determined for
an on-process block Gauss--Seidel relaxation which exactly inverts the
advective component in the case of no cycles in the mesh, where the
block size is given by the number of DOFs in a given element face. Such
a relaxation scheme is explored numerically in
\cref{ss:efficiencyZZ}.

\section{Numerical simulations}
\label{sec:results}

This section demonstrates the effectiveness of AIR as a
preconditioner for BiCGSTAB to solve the linear system \cref{eq:systemLambda},
including on moving time-dependent domains (\cref{ss:gaussianpulse}), and
when applying space-time AMR to an interior front problem (\cref{ss:efficiencyZZ}).
All test cases have been implemented in the Modular Finite
Element Method (MFEM) library \cite{mfem-library} with solver support
from HYPRE \cite{hypre-library}. Furthermore, we choose the penalty
parameter in \cref{eq:IPpenalty} as $\alpha = 10p^2$ where $p$ is the
order of the polynomial approximation (see, for example
\cite{Riviere:book}). Unless otherwise specified, AIR is constructed with
distance-one connections for building $R$, with strength tolerance $0.3$;
1-point interpolation \cite{Manteuffel:2019}; no pre-relaxation;
post-forward-Gauss-Seidel relaxation (on process), first on F-points,
followed by all points; Falgout coarsening, with strength tolerance $0.2$;
and as an acceleration method for BiCGSTAB, applied to the HDG space-time
matrix, scaled on the left by the facet block-diagonal inverse. All parallel
simulations are run on the LLNL Quartz machine.

\subsection{Rotating Gaussian pulse on a time-dependent domain}
\label{ss:gaussianpulse}

We first consider the solution of a two-dimensional rotating Gaussian
pulse on a time-dependent domain \cite{Rhebergen:2013}. We set
$a = (-4x_2, 4x_1)^T$ and $f = 0$. The boundary and initial conditions
are chosen such that the analytical solution is given by
\begin{equation}
  u(t, x_1, x_2) = \frac{\sigma^2}{\sigma^2 + 2 \nu t}\exp\left(-\frac{(\tilde{x}_1 - x_{1c})^2 + (\tilde{x}_2 - x_{2c})^2}
    {2\sigma^2 + 4 \nu t}\right),
\end{equation}
where $\tilde{x}_1 = x_1 \cos(4t) + x_2 \sin(4t)$,
$\tilde{x}_2 = -x_1 \sin(4t) + x_2 \cos(4t)$, and
$(x_{1c}, x_{2c}) = (-0.2, 0.1)$. Furthermore, we set $\sigma = 0.1$
and consider both a diffusion-dominated case with $\nu = 10^{-2}$ and
an advection-dominated case with $\nu = 10^{-6}$. The deformation of
the time-dependent domain is based on a transformation of the uniform
space-time mesh $(t, x_1^u, x_2^u) \in [0, T] \times [-0.5, 0.5]^2$
given by
\begin{equation}
  x_i = x_i^u + A(\tfrac{1}{2}-x_i^u)\sin(2\pi(\tfrac{1}{2}-x_i^*+t)) \quad i=1,2,
\end{equation}
where $(x_1^*,x_2^*)=(x_2^u,x_1^u)$, $A = 0.1$, and $T$ is the final
time. We show the solution on the time-dependent domain at different
time slices and on the full space-time domain (taking $T=1$) in
\cref{fig:rotatingpulse}.

\begin{figure}[!hbt]
  \centering
  \subfloat[The solution at $t=0.2$.]{\includegraphics[width=0.45\textwidth]{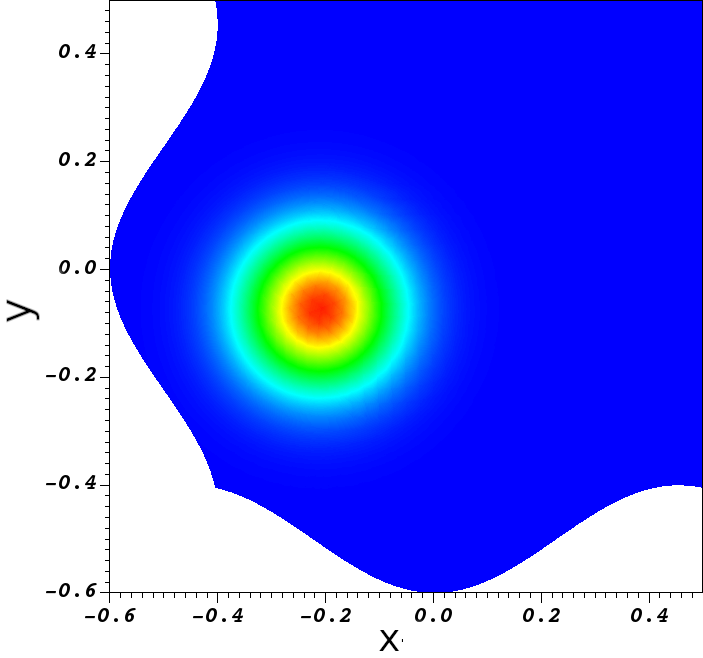}}
  \quad
  \subfloat[The solution at $t=0.5$.]{\includegraphics[width=0.45\textwidth]{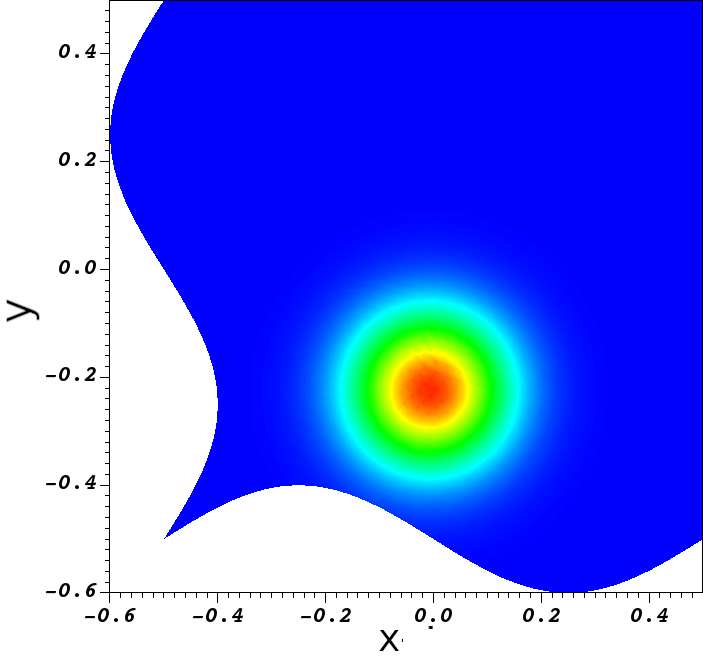}}
  \quad
  \subfloat[The solution at $t=0.8$.]{\includegraphics[width=0.45\textwidth]{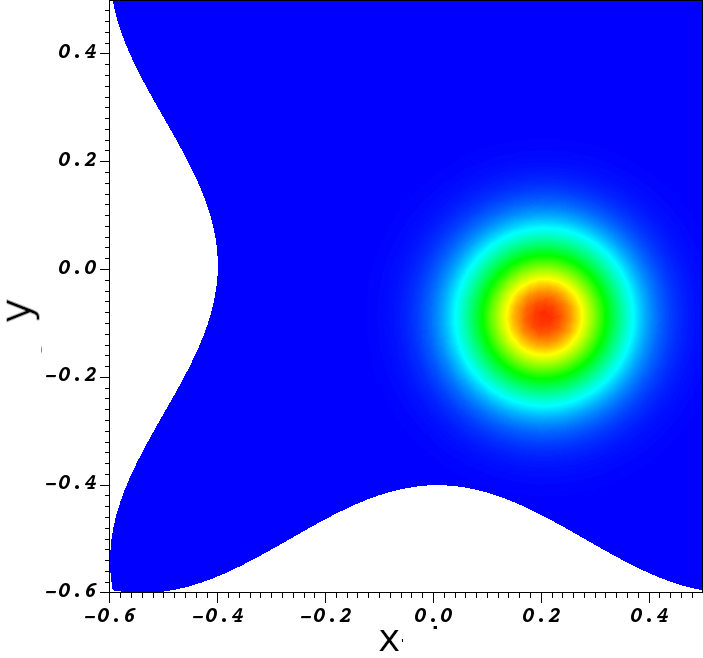}}
  \quad
  \subfloat[The solution over the space-time domain $\mathcal{E}_h$.]{\includegraphics[width=0.45\textwidth]{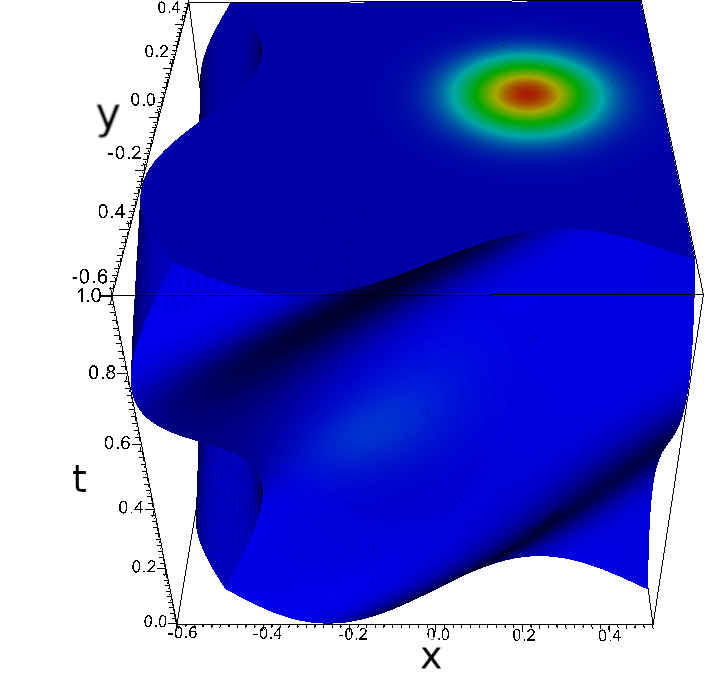}}
  \caption{The solution of the rotating Gaussian pulse test case as
    described in \cref{ss:gaussianpulse} at different time slices
    and on the full space-time domain when $\nu = 10^{-6}$.}
  \label{fig:rotatingpulse}
\end{figure}

\subsubsection*{Rates of convergence of the space-time error}

In \cref{table:slab_conv} we compute the rates of convergence of the
error in the space-time $L^2$-norm, i.e.,
\begin{equation*}
  \norm[0]{u-u_h}_{\mathcal{E}_h}
  := \del{\int_{\mathcal{E}_h} (u-u_h)^2 \dif \hat{x}}^{1/2}.
\end{equation*}
We compute this error taking $T=1$ and using linear, quadratic, and
cubic polynomial approximations to $u$. We observe optimal rates of
convergence, that is, the error in the space-time $L^2$-norm is of order
$\mathcal{O}(h^{p+1})$ when using a $p$-th order polynomial
approximation, for both the all-at-once and slab-by-slab
discretizations. This conclusion is true for both the advection- and
diffusion-dominated problem.

\begin{table}[!htb]
  \centering
  \caption{Error in the space-time $L^2$-norm and rate of
    convergence of a space-time HDG discretization of the
    advection-diffusion problem described in
    \cref{ss:gaussianpulse}, with $T=1$.}
  \label{table:slab_conv}
  \small
  \begin{tabular}{@{}rrrcrrrcrrr@{}}
    \multicolumn{11}{c}{{\bf Slab-by-slab}}\\
    \toprule
    \multicolumn{3}{c}{ } & \multicolumn{2}{c}{$p=1$} & \phantom{abc} & \multicolumn{2}{c}{$p=2$} & \phantom{abc} & \multicolumn{2}{c}{$p=3$} \\
    \cmidrule{4-5} \cmidrule{7-8} \cmidrule{10-11}
    $\nu$ & Slabs & Elements & Error & Rate && Error & Rate && Error & Rate \\
    & & per slab & &  && & &&  &  \\
    \midrule
    \multirow{4}{*}{$10^{-2}$}
    &     8 &    384  & 1.1$e$-2 &   - && 2.9$e$-3 &  -  && 8.4$e$-4 &   -\\
    &    16 &   1,536 & 3.4$e$-3 & 1.7 && 4.7$e$-4 & 2.6 && 5.8$e$-5 &   3.9\\
    &    32 &   6,144 & 8.4$e$-4 & 2.0 && 5.9$e$-5 & 3.0 && 3.7$e$-6 &   4.0\\
    &    64 &  24,576 & 2.1$e$-4 & 2.0 && 7.4$e$-6 & 3.0 && 2.3$e$-7 &   4.0\\
    \\
    \multirow{4}{*}{$10^{-6}$}
    &     8 &    384  & 1.9$e$-2 &  -  && 5.3$e$-3 &  -  && 1.3$e$-3 & -\\
    &    16 &   1,536 & 6.1$e$-3 & 1.6 && 8.5$e$-4 & 2.7 && 1.3$e$-4 & 3.4\\
    &    32 &   6,144 & 1.6$e$-3 & 2.0 && 1.1$e$-4 & 3.0 && 9.0$e$-6 & 3.8\\
    &    64 &  24,576 & 3.8$e$-4 & 2.0 && 1.4$e$-5 & 3.0 && 5.9$e$-7 & 3.9\\
    \bottomrule
    \multicolumn{11}{c}{} \\
    \multicolumn{11}{c}{{\bf All-at-once}} \\
    \toprule
    \multicolumn{2}{c}{ } & \phantom{abc} & \multicolumn{2}{c}{$p=1$} & \phantom{abc} & \multicolumn{2}{c}{$p=2$} & \phantom{abc} & \multicolumn{2}{c}{$p=3$} \\
    \cmidrule{4-5} \cmidrule{7-8} \cmidrule{10-11}
    $\nu$ & Elements && Error & Rate && Error & Rate && Error & Rate \\
    \midrule
    \multirow{4}{*}{$10^{-2}$}
    &  2,760   && 2.0$e$-2 & -   && 6.0$e$-3 &  -  && 1.7$e$-3 &   -\\
    & 22,080   && 5.8$e$-3 & 1.8 && 8.2$e$-4 & 2.9 && 1.2$e$-4 & 3.8\\
    &176,640   && 1.3$e$-3 & 2.1 && 9.5$e$-5 & 3.1 && 7.5$e$-6 & 4.0\\
    &1,413,120 && 3.0$e$-4 & 2.1 && 1.2$e$-5 & 3.1 && 4.6$e$-7 & 4.0\\
    \\
    \multirow{4}{*}{$10^{-6}$}
    &  2,760  && 5.5$e$-2 & -   && 2.1$e$-2 &  -  && 8.9$e$-3 &   -\\
    & 22,080  && 2.0$e$-2 & 1.4 && 3.6$e$-3 & 2.6 && 7.2$e$-4 & 3.6\\
    &176,640  && 5.1$e$-3 & 2.0 && 3.8$e$-4 & 3.2 && 4.5$e$-5 & 4.0\\
    &1,413,120&& 1.0$e$-3 & 2.3 && 4.3$e$-5 & 3.2 && 2.6$e$-6 & 4.1\\
    \bottomrule
  \end{tabular}
\end{table}

\subsubsection*{Performance of BiCGSTAB with AIR as preconditioner}

This section demonstrates the performance of BiCGSTAB with AIR as a
preconditioner in both the advection- and diffusion-dominated
regimes. We will use the number of iterations to convergence as the
indicator of performance as we know that the setup time and cost of
applying AIR per BiCGSTAB iteration are linear with respect to the
matrix size, i.e. $\mathcal{O}(N)$ \cite{Manteuffel:2018,
  Manteuffel:2019}. Hence, the total cost to solve the space-time HDG
problem will be linearly dependent on the number of BiCGSTAB
iterations. In \cref{table:iterations} we list the total number of
BiCGSTAB iterations that are required to reach a relative residual of
$10^{-12}$ in an all-at-once discretization with $T=1$ and using
linear, quadratic, and cubic polynomial approximations to $u$.

\begin{table}[tbp]
  \centering
  \caption{The number of BiCGSTAB iterations (with AIR as the
    preconditioner) required to reach a relative residual of
    $10^{-12}$ for the test case described in
    \cref{ss:gaussianpulse} with $T=1$.
    The stopping tolerance was not reached
    within 5000 iterations if a value is missing.}
  \label{table:iterations}
  \begin{tabular}{@{}crrrrr@{}}
    \multicolumn{6}{c}{$\boldsymbol{p=1}$} \\
    \toprule
    \multirow{2.5}{*}{\begin{tabular}{@{}c@{}}\\DOFs\end{tabular}} & \multicolumn{5}{c}{$\nu$} \\
    \cmidrule{2-6} & $10^{-6}$ & $10^{-4}$ & $10^{-3}$ & $10^{-2}$ & $10^{-1}$ \\
    \midrule
    17,496    & 7 & 7 & 7 & 8 & 12 \\
    136,224   & 8 & 7 & 8 & 10 & 17\\
    1,074,816 & 8 & 8 & 10 & 13 & 54\\
    8,538,624 & 8 & 9 & 12 & 18 & -\\
    \bottomrule
    \multicolumn{6}{c}{ } \\
    \multicolumn{6}{c}{$\boldsymbol{p=2}$} \\
    \toprule
    \multirow{2.5}{*}{\begin{tabular}{@{}c@{}}\\DOFs\end{tabular}} & \multicolumn{5}{c}{$\nu$} \\
    \cmidrule{2-6} & $10^{-6}$ & $10^{-4}$ & $10^{-3}$ & $10^{-2}$ & $10^{-1}$ \\
    \midrule
    34,992     & 11 & 8 & 9 & 11 & 19 \\
    272,448    & 8 & 9 & 10 & 14 & 30 \\
    2,149,632  & 9 & 11 & 13 & 18 & 46\\
    17,077,248 & 9 & 14 & 15 & 30 & 83\\
    \bottomrule
    \multicolumn{6}{c}{ } \\
    \multicolumn{6}{c}{$\boldsymbol{p=3}$} \\
    \toprule
    \multirow{2.5}{*}{\begin{tabular}{@{}c@{}}\\DOFs\end{tabular}} & \multicolumn{5}{c}{$\nu$} \\
    \cmidrule{2-6} & $10^{-6}$ & $10^{-4}$ & $10^{-3}$ & $10^{-2}$ & $10^{-1}$ \\
    \midrule
    58,320     & 9 & 8 & 10 & 14 & 26 \\
    454,080    & 9 & 11 & 12 & 18 & 38 \\
    3,582,720  & 9 & 13 & 15 & 25 & 73\\
    28,462,080 & 10 & 17 & 18 & 46 & 144\\
    \bottomrule
  \end{tabular}
\end{table}

When the problem is close to hyperbolic (when $\nu = 10^{-6}$) we
observe perfect scalability, that is, the number of iterations
required to converge does not change with the problem size. In the
advection-dominated regime, $\nu = 10^{-4}$ and $\nu = 10^{-3}$, the
iteration count increases slightly with problem size, but the increase
is slow, the iteration counts remain quite low. When more significant
diffusion is introduced, $\nu = 10^{-2}$ and $\nu = 10^{-1}$, the
iteration count starts to grow more rapidly with increasing problem
size. These observations hold for all polynomial degrees
considered. It is worth pointing out that for $\nu = 10^{-2}$ and
$\nu = 10^{-1}$, using a classical $P^TAP$ AMG approach rather than
AIR did result in lower iteration counts (not shown), however, the
total time to solution remained notably longer than that of AIR,
likely due to denser coarse-grid matrices.

From the above observations, we may conclude that BiCGSTAB with AIR as
the preconditioner is an excellent iterative solver for the solution
of all-at-once space-time HDG discretizations of the
advection-diffusion problem in the advection-dominated
regime. Unsurprisingly, the solver is suboptimal in the
diffusion-dominated regime. To see why, note that we may write
\cref{eq:stadvdif_a} as
\begin{equation*}
  \widehat{a} \cdot \widehat{\nabla} u - \widehat{\nabla} \cdot (\widehat{\nu} \widehat{\nabla} u)
  = f  \qquad \text{in } \mathcal{E}_h,
\end{equation*}
where $\widehat{\nu} = \text{diag}(0,\nu,\nu)$ (note that there is no
diffusion in the time direction). This is a ``steady''
advection-diffusion problem in $(d+1)$-dimensional space-time with
completely \emph{anisotropic} diffusion in $d$ dimensions and
advection in one dimension. Problems with anisotropic diffusion are
known to pose a challenge to multilevel solvers (see, for example,
\cite{Schroder:2012} for a literature review on the challenges of
using multilevel solvers for problems with anisotropic
diffusion). Robust solvers for the mixed regime of advection and
strongly anisotropic diffusion are ongoing work.

\begin{remark}[Stopping tolerance]
  \label{rem:stopping}
  Our main goal is to test the performance of BiCGSTAB with AIR as a
  preconditioner. For this reason, we chose the stopping criteria for
  BiCGSTAB to be a relative residual of $10^{-12}$. In practice,
  however, the stopping tolerance need not be chosen this small (see,
  for example \cite[pages 73, 77--79]{Elman:Book}). We demonstrate
  this also in \cref{table:l2bicgs} where, for the case $p=2$ for a
  problem with 272,448 degrees of freedom. We note that the error in
  the space-time $L_2$-norm does not improve after the first full
  BiCGSTAB iteration although it takes 8 iterations to reach a
  relative residual of $10^{-12}$ (see \cref{table:iterations}).
\end{remark}

\begin{table}[!htbp]
  \centering
  \caption{Error in the space-time $L_2$-norm as a
    function of BiCGSTAB iteration number for the test case from
    \cref{table:iterations}. We use a quadratic ($p=2$) polynomial
    approximation and the linear system has 272,448 degrees of
    freedom. The preconditioned residual presented in the table is the
    residual of the full-step.}
  \label{table:l2bicgs}
  \begin{tabular}{@{}ccc@{}}
    \toprule
    Iteration & Preconditioned & \multirow{2}{*}{$\dnorm{u-u_h}_{L_2(\mathcal{E}_h)}$} \\
    Number & Residual & \\
    \midrule
    0 & 2.5e-4 & 1.6e-3\\
    1 & 8.4e-6 & 3.6e-3\\
    2 & 1.0e-6 & 3.6e-3\\
    3 & 6.1e-9 & 3.6e-3\\
    4 & 1.6e-10 & 3.6e-3\\
    \bottomrule
  \end{tabular}
\end{table}

\begin{remark}[GMRES and other Krylov]
  In this section we considered the performance of BiCGSTAB with AIR
  as a preconditioner. Of course, we may replace BiCGSTAB with any
  other iterative method for non-symmetric systems of linear
  equations. GMRES performed equally well in most cases; however,
  there were several examples that stalled significantly upon GMRES
  restart, and which also required a moderately high number of
  iterations to convergence, limiting the use of full-memory GMRES.
  In our tests, BiCGSTAB has appeared to be slightly more robust and,
  thus, is used for all numerical tests presented here.
\end{remark}

\subsubsection*{Scalability and parallel-in-time on moving domains}

The current paradigm in scientific computing is to use multiple
computing units simultaneously to lower runtime. Hence, the
scalability of an algorithm is an important measure of
performance. Ideally, the runtime should be inversely proportional to
the number of computing units. Unfortunately, this is not always
achievable due to limited fast access memory (caches), limited memory
bandwidth, and inter-process and inter-node communication. One
advantage of the all-at-once space-time approach over the slab-by-slab
space-time approach is the better communication to computation
ratio. This is because it is possible to parallelize both in space and
time simultaneously, as opposed to the standard parallel-in-time
approach of treating space and time separately.

To test the scalability of BiCGSTAB with AIR as a preconditioner
applied to the space-time HDG discretization, we will measure the
total wall-clock time spent on solving the rotating Gaussian pulse
problem discussed at the beginning of this section. For this, we
consider a final time of $T=16$, we consider both an advection-
($\nu = 10^{-6}$) and a diffusion-dominated ($\nu=10^{-2}$) problem,
and we consider both an all-at-once and a slab-by-slab
discretization. For the all-at-once discretization, we consider two
unstructured space-time meshes; the coarse mesh consists of 45576
tetrahedra and the fine mesh consists of 364608 tetrahedra. For the
slab-by-slab approach, we consider a coarse mesh in which the
space-time domain is divided into 128 space-time slabs and each slab
consists of 384 tetrahedra. The fine slab-by-slab mesh consists of 256
slabs and each slab consists of 1536 tetrahedra. Note that the
slab-by-slab meshes were created to have a similar number of
tetrahedra as the all-at-once space-time meshes.

The total wall-clock times we measure are the combination of time
spent on the following four stages: setup, assembly, solving, and
reconstruction. During the setup stage, the mesh is read from a file
and refined sequentially and finite element spaces and linear and
bi-linear forms are created. We remark that this stage is not
parallelizable and it affects the speedup we obtain. The assembly
stage contains the computation of elemental matrices, computation of
elemental Schur complements, and the assembly of the global linear
system \cref{eq:systemLambda}. This stage is almost embarrassingly
parallel. The next stage is the solve stage in which the global linear
system is solved using BiCGSTAB with AIR as the preconditioner. This
stage is weakly scalable. Finally, the element solution
$U = A^{-1}(F-B\Lambda)$ is reconstructed in the reconstruction stage
(see \cref{ss:discretization}). This step, in theory, does not require
any communication as it can be done completely locally.

Parallel speedup in a strong-scaling sense for each combination of
mesh resolution and diffusion coefficient is shown in
\cref{fig:scalability}.  We see that, in all cases, the all-at-once
approach is the best algorithm sequentially. Hence, the speedups are
calculated relative to the sequential timing of the solutions using
all-at-once approach for different order of approximations.  The best
speedup we achieve at 256 processes is slightly more that 100, and
just less than 50\% efficiency. This can be mostly attributed to the
sequential nature of the setup stage, for example, it takes up to 10\%
of wall-clock time spent for large problems solved with many (64-256)
cores. In addition to this, there is a significant loss of scalability
during the solve stage, which is largely due to the algorithm becoming
communication bound and thereby less efficient in parallel. For
example, the speedup observed on the fine mesh at 256 processes is
close to $2\times$ larger than that observed on the coarse
mesh. Hence, for larger problems, we expect better speedup with the
primary bottleneck being the setup stage. It is worth pointing out
that a number of recent works have developed architecture-aware
communication algorithms for sparse matrix-vector operations and AMG
that can significantly improve scalability in the communication-bound
regime (e.g., \cite{bienz2019node,bienz2020reducing}), but we do not
exploit such methods here.

\cref{fig:relspeedup} plots the relative speedup of the
all-at-once approach to the slab approach with respect to
wall-clock time, that is,
$\text{Time}_{\text{all-at-once}}(\text{n})/$$\text{Time}_{\text{slab-by-slab}}(\text{n})$.
We see that, generally, the all-at-once approach is 20\% to 50\%
faster than the slab-by-slab approach, although in some cases it is up
to $2\times$ faster.  Note that this comparison is imperfect, and an
accurate measure of speedup is nuanced -- for example, the slab mesh
here has roughly 8\% more elements than the all-at-once mesh; however,
the slab mesh is also structured in time, while the all-at-once mesh
is fully unstructured in space and time, which can degrade performance
of multigrid solvers on a fine-grained/memory-access level.  They also
differ algorithmically; for example, the all-at-once approach does one
setup phase for AIR, followed by the solve phase, while using the
slab-by-slab approach requires rebuilding the solver each time
step. In general, we do not try to isolate where the speedup comes
from in this paper. Rather, we highlight here that by using AIR as a
full space-time solver, we are able to see speedups over sequential
time stepping for low core counts, a property that is not shared
by most parallel-in-time schemes.

\begin{figure}
  \centering
  \includegraphics[width=0.45\linewidth]{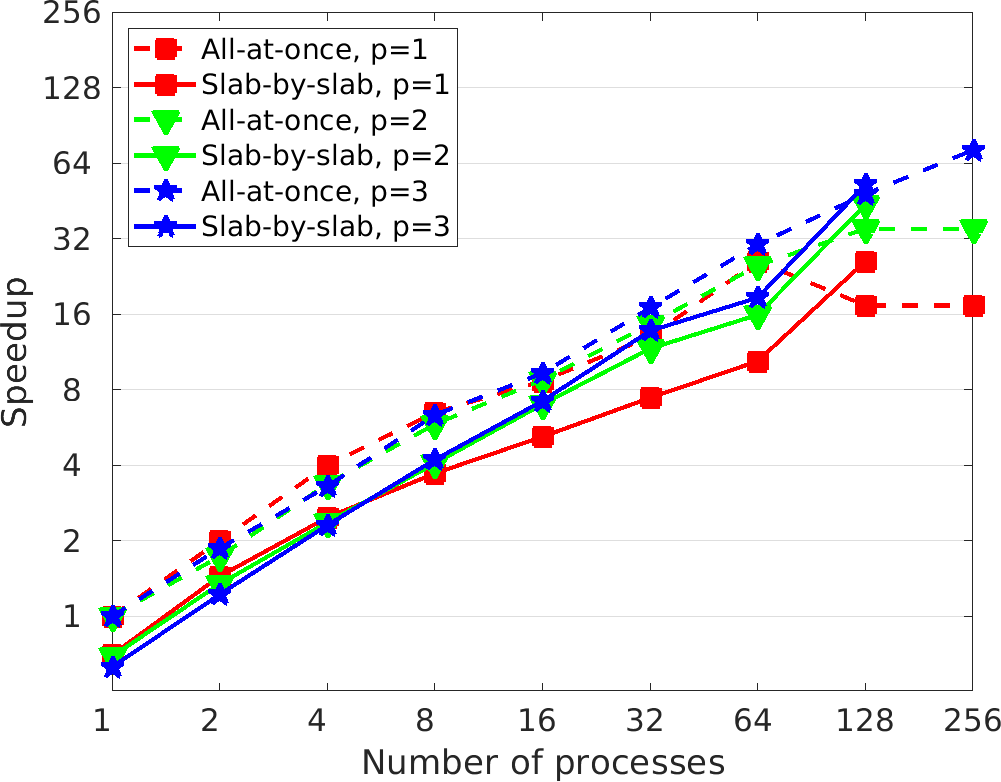}
  \includegraphics[width=0.45\linewidth]{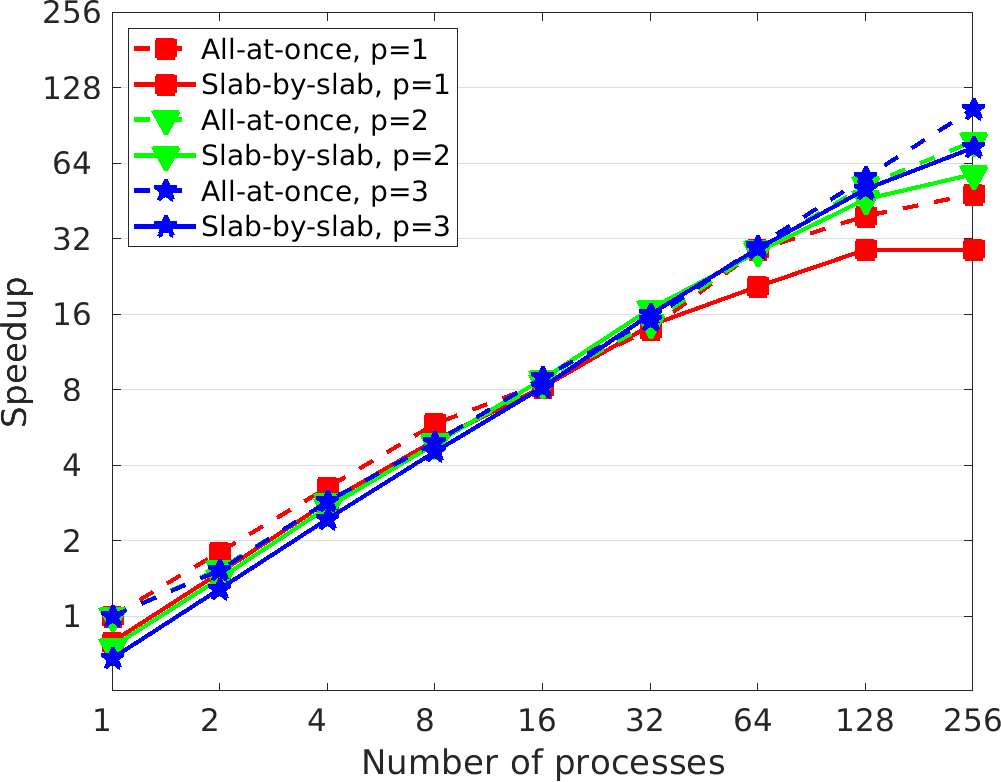}
  \\
  \includegraphics[width=0.45\linewidth]{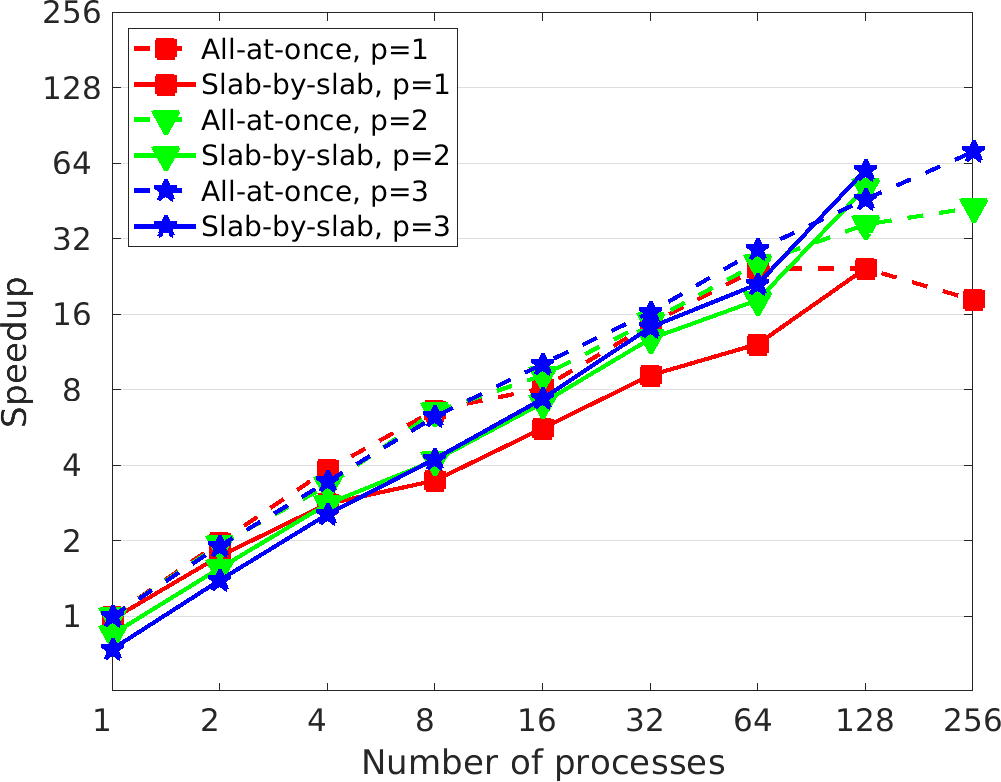}
  \includegraphics[width=0.45\linewidth]{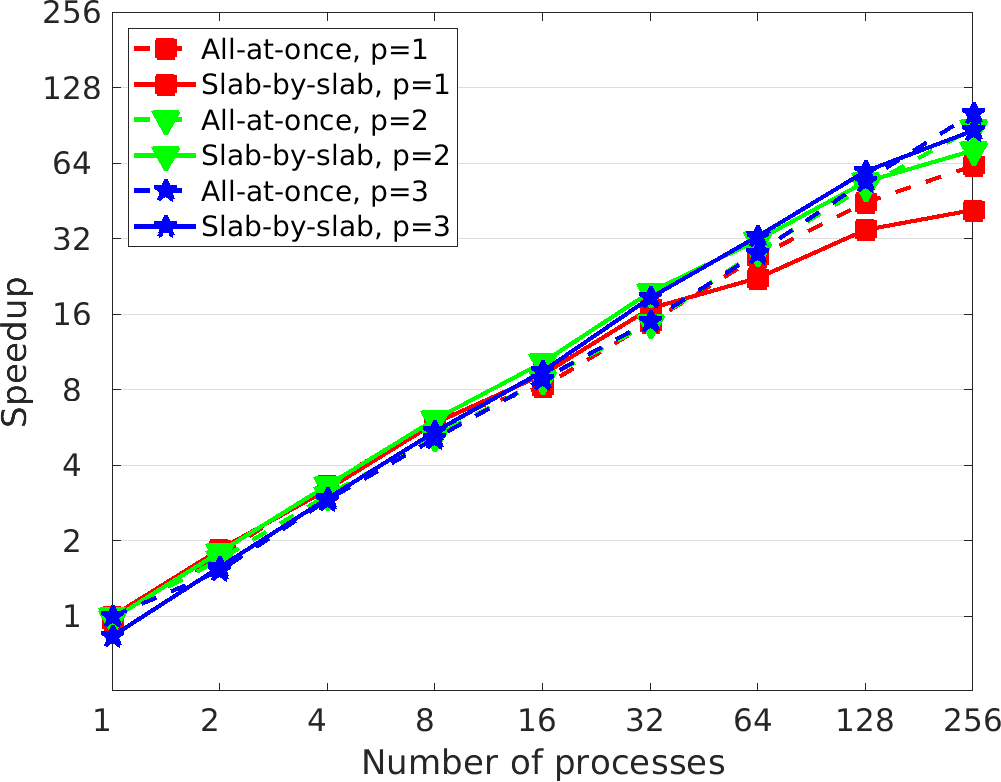}
  \caption{Parallel scalability. Top left: $\nu=10^{-6}$, coarse mesh,
    top right: $\nu=10^{-6}$, fine mesh, bottom left: $\nu=10^{-2}$,
    coarse mesh, bottom right: $\nu=10^{-2}$, fine mesh}
  \label{fig:scalability}
\end{figure}

\begin{figure}
  \centering
  \includegraphics[width=0.45\linewidth]{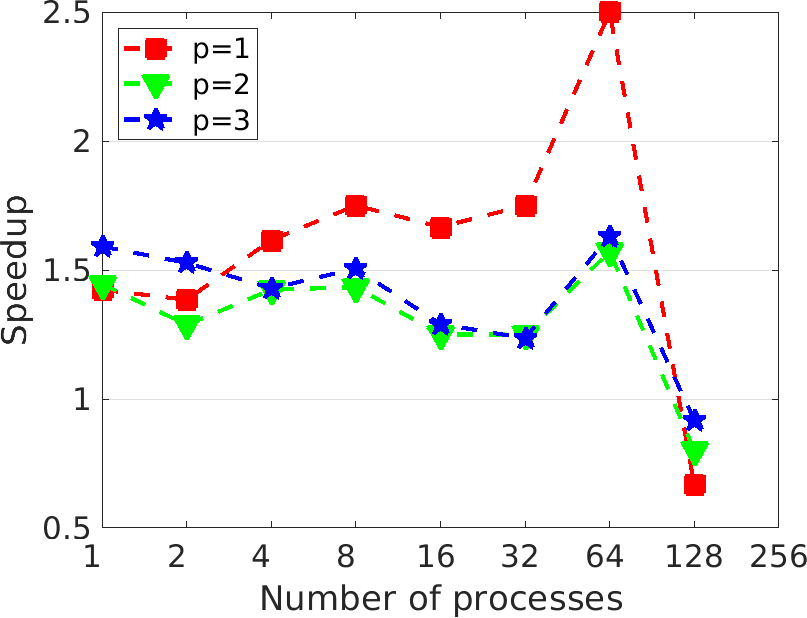}
  \includegraphics[width=0.45\linewidth]{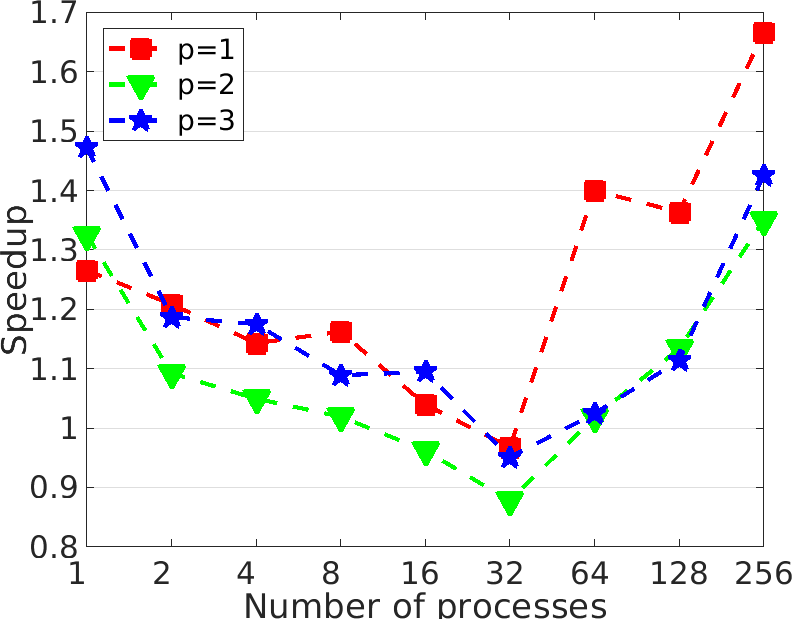}
  \\
  \includegraphics[width=0.45\linewidth]{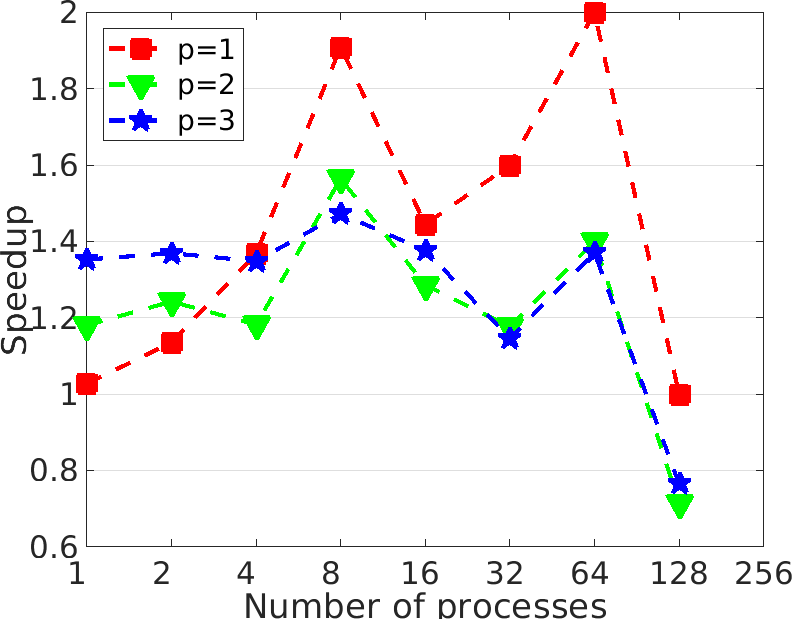}
  \includegraphics[width=0.45\linewidth]{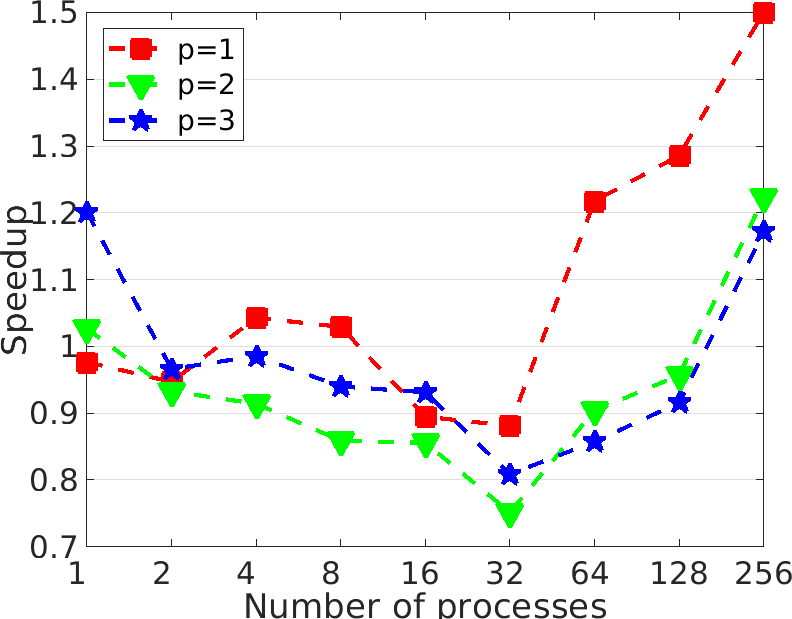}
  \caption{Relative speedup of all-at-once approach against
    slab-by-slab approach. Top left: $\nu=10^{-6}$, coarse mesh, top
    right: $\nu=10^{-6}$, fine mesh, bottom left: $\nu=10^{-2}$,
    coarse mesh, bottom right: $\nu=10^{-2}$, fine mesh}
  \label{fig:relspeedup}
\end{figure}

\subsection{Moving internal layer problem}
\label{ss:efficiencyZZ}

We now consider the moving internal layer problem proposed in
\cite{Frutos:2014}. We solve \cref{eq:stadvdif} on the unit cube
space-time domain ($\mathcal{E}_h = [0,1]^3$) with $a=(x_2, -x_1)^T$,
$f=0$, and with $\nu = 0$ (the hyperbolic limit). We impose a Neumann
boundary at $t=0$, on which we set $g_N = 0$, and an outflow boundary
at the final time $t=1$. On the boundary $x_2=0$ we set $g_D = 1$ and
we set $g_D=0$ on the remaining boundaries. For the time interval of
interest, the exact solution is given by
\begin{equation*}
  u(t,x_1,x_2) =
  \begin{cases}
    1 & \text{when } \norm[0]{(x_1,x_2)}_2<1 \text{ and } \atantwo(x_2,x_1) > \pi/2-t, \\
    0 & \text{otherwise},
  \end{cases}
\end{equation*}
which describes a front that rotates around the origin as time evolves.

\subsubsection*{Space-time adaptive mesh refinement}

To efficiently solve this problem, we use space-time adaptive mesh
refinement (AMR) in an all-at-once discretization, where we refine
locally in both space and time. The Zienkiewicz--Zhu (ZZ) error
estimator \cite{ZZ:book, Zienkiewicz:1987, Zienkiewicz:1992} is used
to mark space-time elements that need to be refined. Although the ZZ
error estimator is not theoretically efficient or reliable for many
problems, it is often used heuristically in adaptive finite element
codes due to its simplicity, low computational cost, and wide
availability. See \cref{fig:interior} for a plot of the mesh and the
solution at two different time slices. A plot of the adaptively
refined mesh in space-time is given in \cref{fig:amr-mesh-plots}.
\begin{figure}
  \centering
  \subfloat[Mesh slice at $t=0.5$]{
  \includegraphics[width=0.46\linewidth]{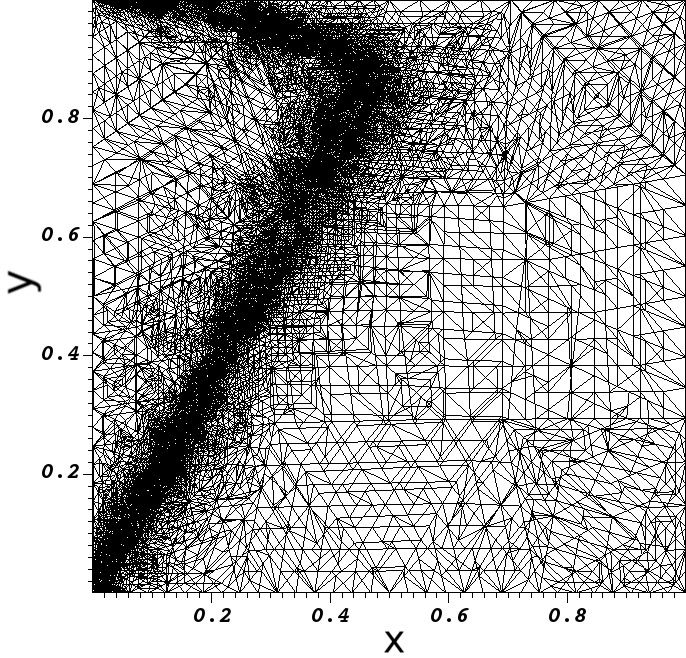}
  }
  \hfill
  \subfloat[Solution slice at $t=0.5$]{
  \includegraphics[width=0.5\linewidth]{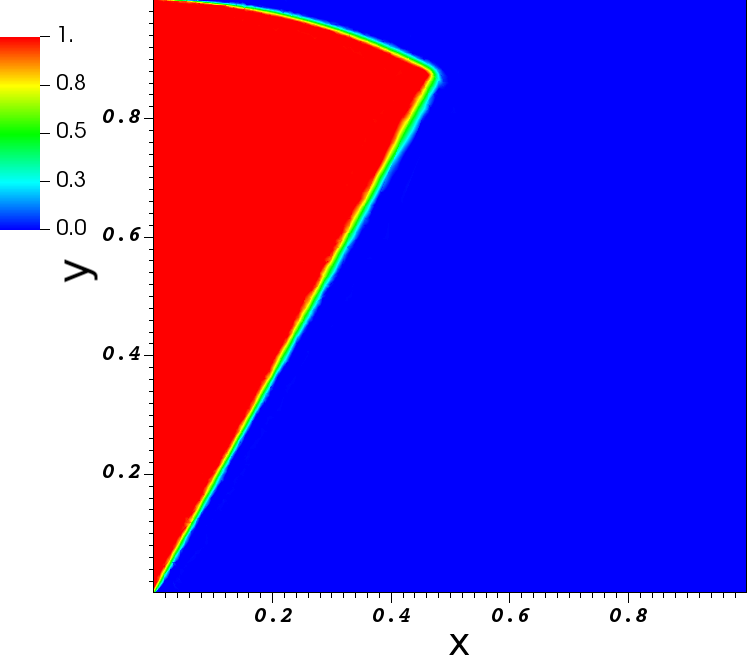}
  }
  \\
  \subfloat[Mesh slice at $t=1$]{
  \includegraphics[width=0.46\linewidth]{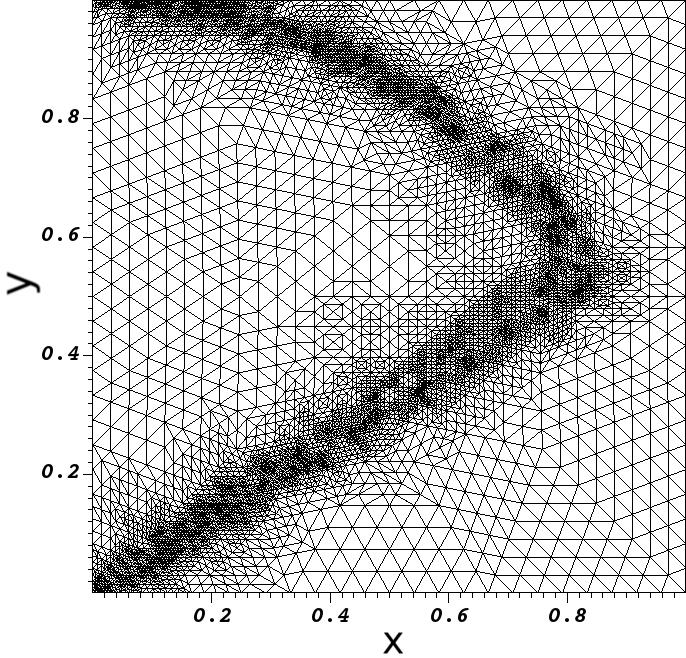}
  }
  \hfill
  \subfloat[Solution slice at $t=1$]{
  \includegraphics[width=0.5\linewidth]{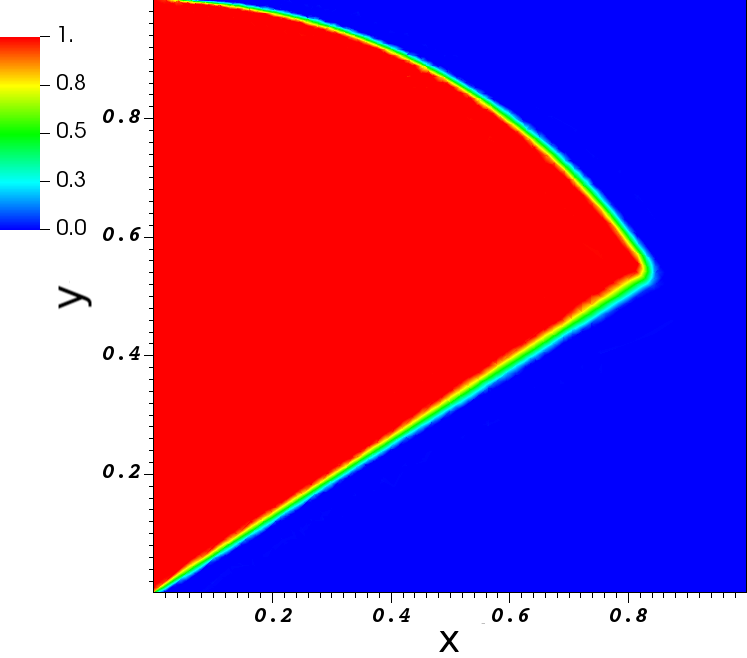}
  }
  \caption{The numerical solution to the interior layer problem at two
    different time slices. The non-triangular polygons in the top left
    figure are because we are slicing the space-time mesh at $t=0.5$;
    we are cutting through space-time tetrahedra.}
  \label{fig:interior}
\end{figure}
\begin{figure}
  \centering
  \includegraphics[width=0.46\linewidth]{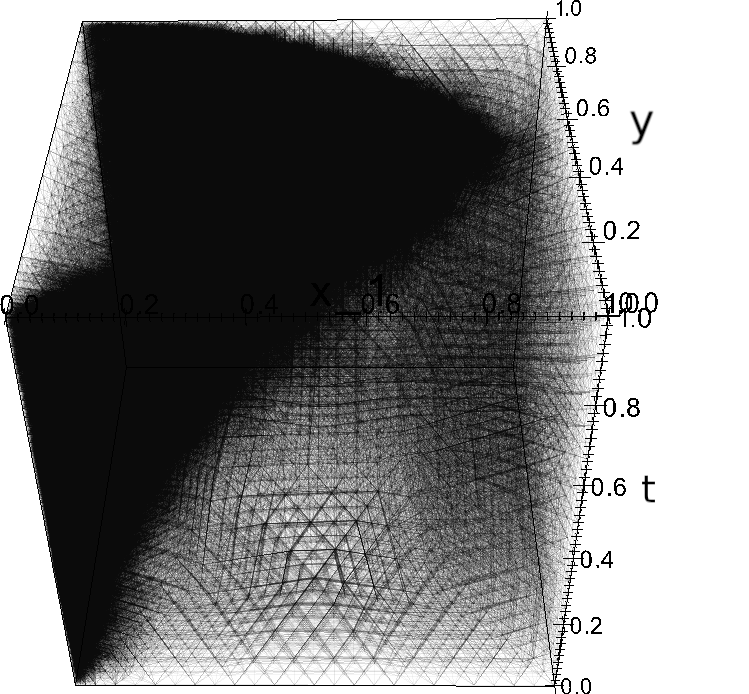}
  \qquad
  \includegraphics[width=0.46\linewidth]{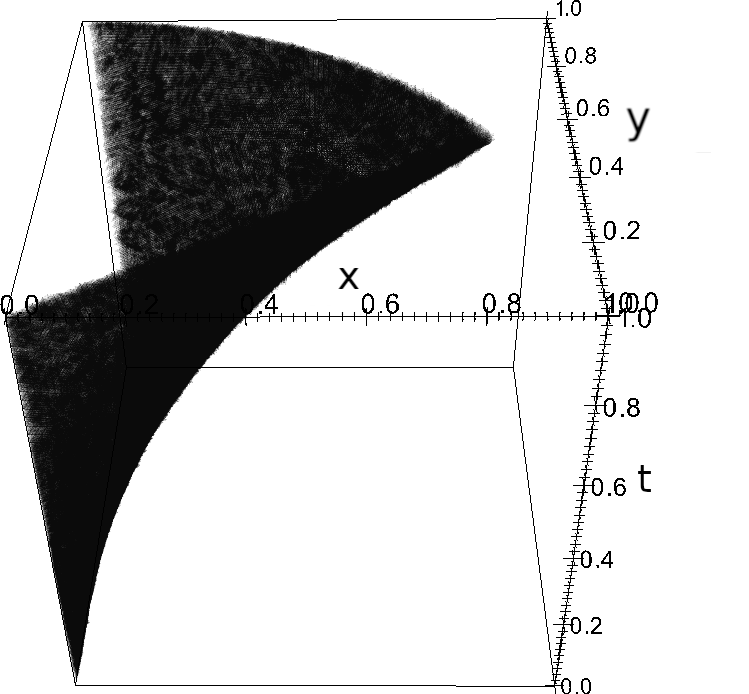}
  \caption{Left: the space-time AMR mesh obtained using the ZZ error
    estimator for the test case described in \cref{ss:efficiencyZZ}.
    Right: only the elements below the median element size are
    shown. Note that the mesh is refined along the space-time interior
    layer.}
  \label{fig:amr-mesh-plots}
\end{figure}

In \cref{fig:adaptivityplot} we compare the convergence of the error
in the space-time $L^2$-norm using space-time AMR to using uniformly
refined meshes.
\begin{figure}
  \centering
  \includegraphics[width=0.65\linewidth]{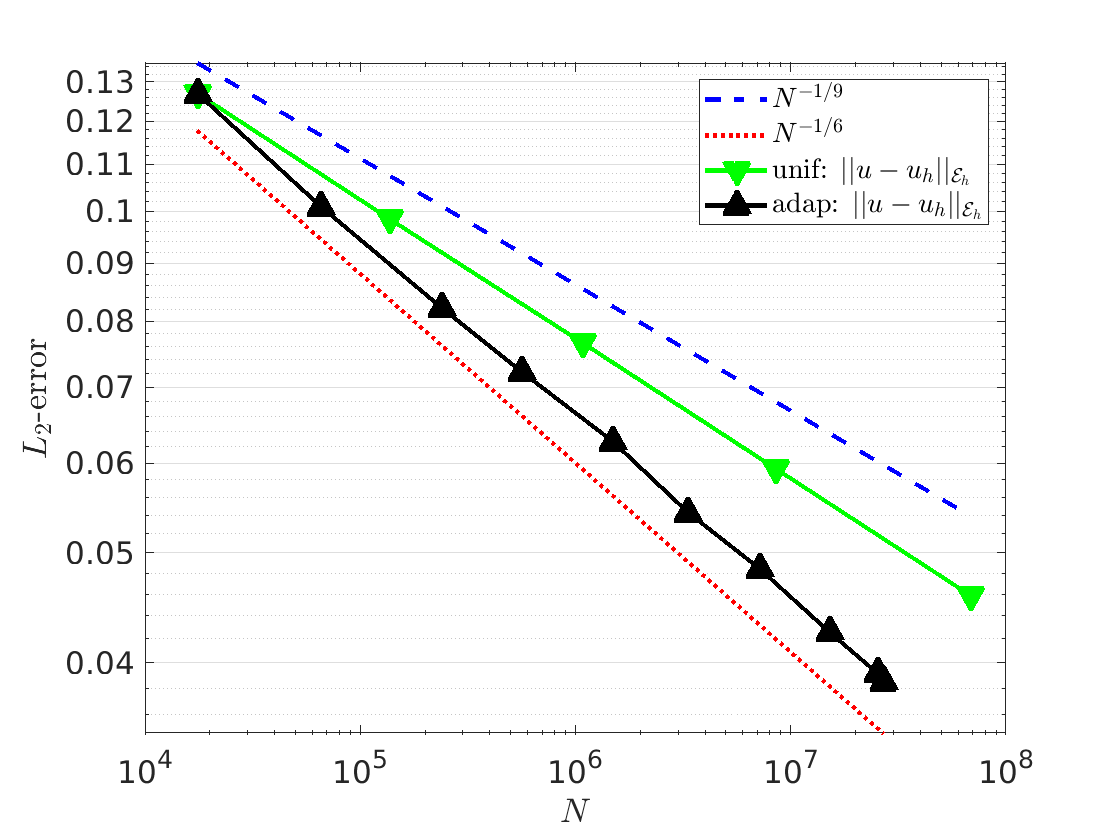}
  \caption{We compare the convergence of the error in the space-time
    $L^2$-norm using space-time AMR to using uniformly refined
    space-time meshes. The test case is described in
    \cref{ss:efficiencyZZ}. Here $N$ is the total number of globally
    coupled DOFs.}
  \label{fig:adaptivityplot}
\end{figure}
Let $N$ be the total number of globally coupled DOFs. We see that the
error on uniform meshes is approximately $\mathcal{O}(N^{-1/9})$ while
the error on the space-time AMR meshes is approximately
$\mathcal{O}(N^{-1/6})$, i.e., we obtain faster convergence using
space-time AMR than when using uniformly refined meshes. We remark
that the error when using an efficient and reliable error estimator
for this problem is expected to be $\mathcal{O}(N^{-1/3})$ (see, for
example, \cite{Burman:2009} for the analysis of an a posteriori error
estimator for a DG discretization of the steady advection equation).
However, we are not aware of any efficient and reliable error
estimators for space-time HDG discretizations of the time-dependent
advection equation.

\subsubsection*{Performance of AIR and on-process solves}

Last, we consider the performance of BiCGSTAB with AIR as a
preconditioner within the context of space-time AMR, and demonstrate
the application of \cref{lem:order}. It is well known that upwind DG
discretizations of advection on convex elements yield matrices that
are block triangular in some element ordering. This can serve as a
robust on-process relaxation routine, where the triangular element
ordering is obtained and an ordered block Gauss--Seidel exactly
inverts the on-process subdomain \cite{hanophy2020parallel}.  AIR
also relies, in some sense, on having a matrix with dominant lower
triangular structure, where it can be shown that triangular structure
allows for a good approximation to ideal restriction
\cite{Manteuffel:2019}. Although HDG discretizations are not always
thought of as block linear systems in the same way that DG
discretizations are, \Cref{lem:order} proves that by treating DOFs on
a given facet as a block in the matrix, an analogous result holds,
that is, the matrix is block lower triangular in some ordering.  With
AIR preconditioning, the block structure can be accounted for by using
a block implementation of AIR (e.g., \cite{Manteuffel:2018}) coupled
with block relaxation or, in the advection-dominated regime, scaling
on the left by the block-diagonal inverse, wherein the scaled matrix
is then scalar lower triangular.

\Cref{fig:relaxationcompare} demonstrates each of these points in
practice, applying AIR to a succession of adaptively refined
space-time problems with various relaxation and block inverse
strategies. The number of DOFs on the x-axis corresponds to successive
levels of adaptive space-time mesh refinement (with correspondingly
larger number of DOFs).  First, note that accounting for the block
structure in the matrix is important for scalable convergence at
larger problem sizes. Not scaling by the block inverse (``No Block
Inv'') can lead to an increase in iteration count by more than
$3\times$ for the largest problem size, and likely worse as DOFs
further increase. On the other hand, after applying the block inverse
scaling, even pointwise Jacobi relaxation yields near perfectly
scalable convergence. Furthermore, because we are considering a
hyperbolic equation with cycle-free space-time velocity field
$\widehat{a}=(1,x_2, -x_1)^T$, from \cref{lem:order} the scaled matrix
is lower triangular. A topological sort of the on-process matrix
yields the triangular ordering, and an ordered Gauss--Seidel
relaxation then exactly inverts the on-process block.  Simulations
in \cref{fig:relaxationcompare} are run on 128 cores, and we see
that with an on-process solve as relaxation, the number of
iterations required to converge is half of the second-best relaxation
method we tested, forward Gauss--Seidel (although both are still quite
good).\footnote{Note that the moving domain considered in
  \cref{ss:gaussianpulse} introduces cycles in the matrix-graph, and
  the resulting matrix is not necessarily block triangular. However,
  cycle-breaking strategies such as used in \cite{haut2019efficient}
  for DG transport simulations on curvilinear meshes can find a
  ``good'' ordering and provide comparable performance as a direct
  on-process solve when coupled with the larger AIR algorithm.}

\begin{remark}[Relation to PinT]
  It is worth pointing out the relation of the on-process solve to
  PinT methods. In MGRiT and Parareal, the relaxation scheme
  corresponds to solving the time-propagation problem between
  F-time-points. If you assign one F-point per process, this is
  solving the time-propagation problem exactly on process and is
  coupled with a coarsening in time.  Similar to the discussion in
  \cref{ss:coarsen} here, we actually solve the space-time problem
  exactly \textit{along the characteristics} on each process, which is
  then coupled with a coarsening that aligns with the characteristics
  (see \cref{fig:CFpoints}). Again, we believe this more holistic
  treatment of space and time is what allows for perfectly scalable
  parallel-in-time convergence on hyperbolic problems.
\end{remark}

\begin{figure}
  \centering
  \includegraphics[width=0.65\linewidth]{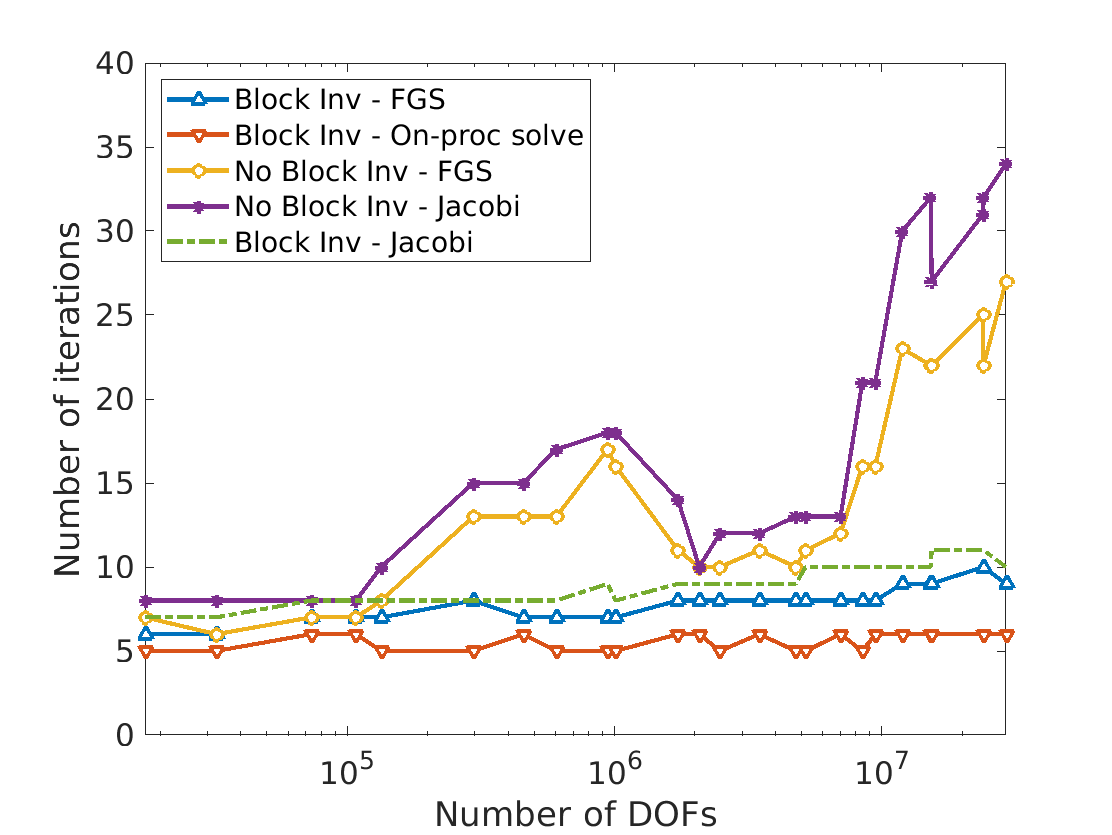}
  \caption{A comparison of the number of BiCGSTAB iterations to
    convergence using AIR as preconditioner with different relaxation
    strategies. We plot the number of iterations against the number of
    globally coupled DOFs at different levels of refinement within the
    AMR algorithm.}
  \label{fig:relaxationcompare}
\end{figure}

\section{Conclusions}
\label{sec:conclusions}

AIR algebraic multigrid is known to be a robust preconditioner for
discretizations of steady advection-dominated advection-diffusion
problems. This paper was motivated by the question whether AIR AMG is
robust as an all-at-once solver for space-time HDG discretizations of
time-dependent advection-diffusion problems, since such problems can
be seen as ``steady'' advection-diffusion problems in
$(d+1)$-dimensions. By numerical examples, we have indeed demonstrated
that AIR provides fast, effective, and scalable preconditioning for
space-time discretizations of advection-dominated problems, including
robust convergence on space-time AMR and moving, time-dependent
domains.

Advection-dominated problems are notoriously difficult for
parallel-in-time methods, motivating a number of efforts to develop
specialized techniques that can handle advection on coarse time-grids
(e.g., \cite{de2019optimizing,dai2013stable}).  Here, we claim that
the best way to provide time parallelism for hyperbolic problems is by
treating space and time \textit{together}. In particular, a critical
component in multigrid methods is constructing an effective coarse
grid. By applying AIR all-at-once to a space-time discretization,
coarsening is able to align with hyperbolic characteristics in
space-time and provide a coarse-grid that naturally captures these
characteristics.  Moreover, we proved that for purely hyperbolic
problems, the space-time HDG discretization on convex elements is
block triangular in some ordering.  Using this ordering, a relaxation
scheme can be designed that exactly solves along the characteristics
on-process, complementing the coarse-grid alignment.  Classical
parallel-in-time multigrid methods that coarsen in space and time
separately are typically unable to align with hyperbolic
characteristics, often resulting in slow convergence or divergence for
time-dependent advection-dominated problems.

\section*{Acknowledgments}

Los Alamos National Laboratory report number LA-UR-20-28396.
This research was enabled in part by the support provided by Sharcnet
(\url{https://www.sharcnet.ca/}) and Compute Canada
(\url{https://www.computecanada.ca}). We are furthermore grateful for
the computing resources provided by the Math Faculty Computing
Facility at the University of Waterloo
(\url{https://uwaterloo.ca/math-faculty-computing-facility/}).

\bibliographystyle{abbrvnat}
\bibliography{references}
\end{document}